\documentclass[11pt]{article}

\usepackage{amsmath,amsthm,amstext,amscd,amssymb,euscript}
\usepackage{mathrsfs}

\usepackage{bm}

\usepackage{epsfig}
\usepackage{url}

\newcommand{\N}{{\mathbb N}}

\newcommand{\R}{{\mathbb R}}
\newcommand{\E}{{\mathbb E}}
\newcommand{\pee}{{\mathbb P}}

\newcommand{\C}{{\EuScript C}}

\def\d{\textup{d}}

\newcommand{\tX}{\tilde{X}}
\newcommand{\tS}{\tilde{S}}

\let\eps=\varepsilon

\let\epsilon=\varepsilon

\newcommand{\mi}{{\mathrm I}}

{2.5}     

\newcommand{\mauvais}{\mathcal{M}}
\newcommand{\tresmauvais}{\widetilde{\mathcal{M}}}
\newcommand{\trestresmauvais}{\widehat{\mathcal{M}}}
\newcommand{\attrac}{\mathscr{A}}

\newcommand{\dimu}{n_{u}}
\newcommand{\dims}{n_{s}}

\newcommand{\Oun}{\mathcal{O}(1)}

\newcommand{\card}{\textup{Card}}

\newcommand{\leb}{\mathrm{Leb}}

\newcommand{\legu}{\Gamma^{u}}
\newcommand{\legs}{\Gamma^{s}}
\newcommand{\cour}{\C_{r,\delta}(x)}

\newtheorem*{maintheorem}{{\small M}{\scriptsize AIN THEOREM}}
\newtheorem{theorem}{{\small T}{\scriptsize HEOREM}}[section]
\newtheorem{corollary}{{\bf{\small C}{\scriptsize OROLLARY}}}[section]
\newtheorem{proposition}{{\bf{\small P}{\scriptsize ROPOSITION}}}[section]
\newtheorem{lemma}{{\bf{\small L}{\scriptsize EMMA}}}[section]

\renewenvironment{proof}[1]
{\noindent{\bf{\small P}{\scriptsize ROOF}}.\hspace{0.1cm} #1} {$\;\blacksquare$\newline}

\def\1{{\mathchoice {\rm 1\mskip-4mu l} {\rm 1\mskip-4mu l}{\rm 1\mskip-4.5mu l} {\rm 1\mskip-5mu l}}}

\begin{document}

\title{Poisson approximation\\ for the number of visits to balls\\
in nonuniformly hyperbolic dynamical systems}
\author{{\bf J.-R. Chazottes, P. Collet}\\
{\small Centre de Physique Th\'eorique}\\
{\small CNRS-\'Ecole Polytechnique}\\
{\small 91128 Palaiseau Cedex, France}}
\date{}

\maketitle

\abstract{
We study the number of visits to balls $B_r(x)$,
up to time $t/\mu(B_r(x))$, for a class of non-uniformly hyperbolic dynamical systems,
where $\mu$ is the SRB measure.
Outside a set of `bad' centers $x$, we prove that this number is
approximately Poissonnian with a controlled error term.
In particular, when $r\to 0$, we get convergence to the Poisson
law for a set of centers of $\mu$-measure one.
Our theorem applies for instance to the H\'enon attractor and, more generally,
to systems modelled by a Young tower whose return-time function has
a exponential tail and with one-dimensional unstable manifolds.
Along the way, we prove an abstract Poisson approximation result of independent
interest.
\bigskip

}

\newpage

\tableofcontents

\section{Introduction and main result}

Consider a discrete-time, ergodic dynamical system $(M,\mu,T)$ where $M$ is a
compact space and $T:M\to M$ is a map preserving the probability measure $\mu$. 
Let $U$ be a subset of $M$. If $\mu(U)>0$, ergodicity ensures that the orbit of
$\mu$-almost every $x\in M$ visits $U$ infinitely many times. Moreover, once an
orbit hits $U$, the time between two consecutive visits is of order $1/\mu(U)$ (this
is a loose interpretation of Ka\v{c} lemma).

We are interested in the distribution of the number of times an orbit
visits a set $U$ with positive measure between time $0$ and $t/\mu(U)$, that is,
the integer-valued random variable 
$$
\sum_{j=0}^{\lfloor t/\mu(U)\rfloor} \1_{\scriptscriptstyle{U}}\circ T^j
$$
on the probability space $(M,\mu)$.

Sets of evident interest are balls $B_r(x)$ of center $x$ and radius $r$ and one
expects that, for ``small'' $r$, the number of visits up to time 
$\lfloor t/\mu(B_r(x))\rfloor$ be approximately distributed according to a Poisson law,
provided correlations decay fast enough and for ``typical'' points $x$.

In the present article, we obtain such a Poisson approximation for a large class of
non-uniformly hyperbolic dynamical systems modelled by a Young tower whose
return-time function has a exponential tail. Postponing the precise definition
of this class to Section \ref{NUDS}, let us state our main theorem. A more precise statement
is given in Theorem \ref{nuhpoisson}. 

\medskip

\begin{maintheorem}
Let  $(M,T,\mu)$ be a non-uniformly hyperbolic dynamical system
modelled by a Young tower whose return-time function has a exponential tail.
Assume that the local unstable manifolds have dimension one.
Denote by $\mu$ its SRB measure. Then there exist constants $C,a,b>0$
such that for all $r\in(0,1)$:
\begin{itemize}
\item There exists a set $\trestresmauvais_r$ such that 
$$
\mu(\trestresmauvais_r)\leq C r^b\, ;
$$
\item For all $x\notin \trestresmauvais_r$ one has
$$
\left|
\mu\left\{y\in M  \; \Bigg| \: 
\sum_{j=0}^{\lfloor t/\mu(B_r(x))\rfloor} \1_{\scriptscriptstyle{B_r(x)}}(T^j y)
=k\right\}-\frac{t^k}{k!}\thinspace e^{-t} \right|\leq C\ r^{a},
$$
for every integer $k\geq 0$ and for every $t>0$. 
\end{itemize}
\end{maintheorem}

\bigskip

\noindent Let us make some comments on this theorem.\newline
The preceding statement immediately implies that, for $\mu$-a.e. center $x$,
\begin{equation}\label{poisson}
\mu\left\{y\in M  \; \Bigg| \: 
\sum_{j=0}^{\lfloor t/\mu(B_r(x))\rfloor} \1_{\scriptscriptstyle{B_r(x)}}(T^j y)
=k\right\}
\xrightarrow{r\to 0}
\frac{t^k}{k!}\thinspace e^{-t}\,. 
\end{equation}
A crucial ingredient in our proof is an estimate of the measure of spherical coronas.
This estimate relies on several general consequences of Besicovitch's covering lemma
of independent interest and seem to be new. This allows us to get explicit estimates
on the error term and on the measure of the set of `bad' centers.\newline
The assumption that unstable manifolds are one-dimensional is likely to be technical.\newline
What we control is in fact the total variation distance between the law of
$\sum_{j=0}^{\lfloor t/\mu(B_r(x))\rfloor} \1_{\scriptscriptstyle{B_r(x)}}\circ T^j$
and the Poisson law, see Theorem \ref{nuhpoisson} below.\newline
The class of dynamical systems we consider was defined in \cite{young1}.
It contains among others Axiom A attractors, the H\'enon attractor for ``good parameters", some dispersing
billiard maps ({\em e.g.}, the periodic Lorentz gas), and piecewise hyperbolic maps of the plane
({\em e.g.}, Lozi attractor).\newline
Let us briefly comment on the results which were available so far. 
There has been a great deal of work in establishing
\eqref{poisson}, and quite often only for $k=0$. Most results were obtained for cylinder sets for some partition,
see {\em e.g.} \cite{av,bv,hirata1,hsv,hv1} and reference therein. The systems considered are
 `mixing' processes on finite alphabets, interval maps, or Axiom A systems. 

There are of course many multidimensional systems for which a Poisson law
is expected. Besides, it is very natural to consider balls
(with respect to the distance on the manifold) instead of cylinders.
Regarding visits to balls for one-dimensional systems ({\em i.e.} intervals),
the first result seems to be found in \cite{pa} for uniformly expanding maps.
Then several types of non-uniformly expanding maps
on the interval  ({\em e.g.} parabolic maps, maps with neutral fixed points)
were handled in \cite{bstv,bv,pierre,colletgalves,fft,hnt}.\newline
In higher dimension,  only a few results are available for balls up to date. 
Dolgopyat \cite{dimitry} established under adequate assumptions a Poisson law for a class of uniformly partially
hyperbolic systems, including Anosov diffeomorphisms. 
Our proof works directly for the case of Axiom A attractors with one-dimensional unstable manifolds.
In \cite{dgs}, the Poisson law is established but only for hyperbolic toral
automorphisms which leave invariant the Haar (Lebesgue) measure.
P\`ene and Saussol \cite{ps} studied return times for the so-called periodic
Lorentz gas with `finite horizon', that is, a planar billiard with periodic configurations of scatterers.
They obtain a convergence in distribution to the exponential law for the rescaled return times to balls.
Finally,  the authors of \cite{ghn} prove convergence towards an exponential law
for balls in certain two-dimensional non-uniformly hyperbolic dynamical systems 
modelled by a Young tower whose return-time function has a exponential tail. But their axioms
do not allow to capture the H\'enon attractor.

\bigskip

\noindent{\bf Content of the article}. 
In Section \ref{generalpoisson} we establish an abstract Poisson approximation bound
for sums of  $\{0,1\}$-valued dependent random variables.  
In Section \ref{NUDS} we describe the class of non-uniformly hyperbolic dynamical systems
we deal with. Then, in Section \ref{preuves} we apply our abstract theorem and control
the error-term for that class of systems. There is an appendix collecting a number
of lemmas, some of them being of general interest.

\section{An abstract Poisson approximation result}\label{generalpoisson}

In the sequel, we denote by $\1_A$ the indicator function of a set $A$.
We recall that if $Y$ and $Z$ are random variables taking integer 
values, their total variation distance is given by
$$
d_{\mathrm{\scriptscriptstyle TV}}\big(Y,Z\big)=\frac12 \sum_{k=0}^\infty
\big|\pee(Y=k)-\pee(Z=k)\big|.
$$
(Strictly speaking, this is a distance between the laws of $Y$ and $Z$ and we should
write $d_{\mathrm{\scriptscriptstyle TV}}\big(\textup{law}(Y),\textup{law}(Z)\big)$.)
By $\textup{Poisson}(\lambda)$ we denote Poisson random variable with mean $\lambda>0$, 
namely
$$
\mathbb{P}(\textup{Poisson}(\lambda)=k)=\frac{\lambda^k}{k!} \ e^{-\lambda}\, .
$$

\begin{theorem}\label{abstract}
Let $(X_n)_{n\in\N}$ be a stationary $\{0,1\}$-valued process and $\eps:=\pee(X_1=1)$.
Then, for all positive integers $p,M,N$ such that $M\leq N-1$ and $2\leq p<N$, one has
\[
d_{\mathrm{\scriptscriptstyle TV}}
\big(X_1+\cdots+X_N,\textup{Poisson}(N\epsilon)\big)
\leq R(\epsilon,N,p,M)
\]
with 
$$
R(\epsilon,N,p,M)=2NM\big[R_{1}(\epsilon,N,p)+R_{2}(\epsilon,p)\big] +R_3(\epsilon,N,p,M)
$$
where
$$
\left\{
\begin{array}{l}
 R_{1}(\epsilon,N,p):= \\
 \sup_{0\le j\le N-p,0\le q\le N-j-p}\left\{\Big|\E\big(\1_{\{X_{1}=1\}} \1_{\{S_{p+1}^{N-j}=q\}}\big)-\epsilon
\E\big( \1_{\{S_{p+1}^{N-j}=q\}}\big)\Big| \right\}\\ \\
R_{2}(\epsilon,p):=\sum_{\ell=1}^{p-1} \E\big(\1_{\{X_{1}=1\}} \1_{\{X_{\ell+1}=1\}}\big)\\ \\
R_3(\epsilon,N,p,M):=4\left(
Mp\epsilon(1+N\epsilon)+ \frac{(\epsilon N)^{M}}{M!}\ e^{- N \epsilon}
+ N\epsilon^{2}\right).
\end{array}\right.
$$
\end{theorem}

\bigskip

\noindent The error term in the above Poisson approximation looks like
the one obtained by the Chen-Stein method \cite{chenstein},
but it involves only future sigma-algebras. In view of applications
to dynamical systems, this is crucial since correlations
(which are related conditional expectations) are in
general controlled only with respect to future sigma-algebras. Here we
use a different method which compares the number of occurrences in a
finite time interval with the number of occurrences in the same interval
for a Bernoulli process $(\tilde{X}_n)$ such that $\pee(\tilde{X}_1=1)=\eps$.
It finally remains to estimate the distance between the number of
occurrences of this Bernoulli process with a Poisson law, but there exists a well
known sharp estimate \cite{lecam}. 

\bigskip

\begin{proof}
Let $(\tX_n)_{n\in\N}$ be a sequence of independent, identically distributed random variables
taking values in $\{0,1\}$. Let $\eps=\pee(\tX_1=1)$ and assume that the $\tX_n$'s
are independent of the $X_n$'s. \newline
We will use the following notations and conventions: $S_i^j:=X_i+X_{i+1}+\cdots+X_j$ and
$\tS_i^j:=\tX_i+\tX_{i+1}+\cdots+\tX_j$, for $1\leq i\leq j$;
$\tS_1^0=S_{N+1}^N:=0$.\newline

\noindent We start by writing a telescoping identity:

\begin{equation}\label{tele}
\pee\big(S_1^N=k\big)-\pee\big(\tS_1^N=k\big) =\sum_{j=0}^{N-1} \Delta_k(j)
\end{equation}
where 
\begin{eqnarray*}
\Delta_k(j) & := & \pee\big(\tS_1^j+S_{j+1}^{N}=k\big)-\pee\big(\tS_1^{j+1}+S_{j+2}^{N}=k\big) \\
                  & = & \sum_{\ell=0}^j \binom{j}{\ell}\ \eps^\ell (1-\eps)^{j-\ell}\ \Phi_{k,j}(\ell),
\end{eqnarray*}
where in turn we set 
$$
\Phi_{k,j}(\ell):=\pee\big(S_{1}^{N-j}=k-\ell\big)-\pee\big(\tX_1+S_{2}^{N-j}=k-\ell\big).
$$
By assumption we have
\begin{eqnarray*}
\lefteqn{\pee\big(\tX_1+S_{2}^{N-j}=k-\ell\big) = }\\
& & \hspace{2cm}(1-\eps)\ \pee\big(S_{2}^{N-j}=k-\ell\big)+\eps \ \pee\big(S_{2}^{N-j}=k-\ell-1\big).
\end{eqnarray*}
Writing
$$
\pee\big(S_{1}^{N-j}=k-\ell\big)=\pee\big(X_1+S_{2}^{N-j}=k-\ell\big) = \pee\big(S_{2}^{N-j}=k-\ell-X_1\big) 
$$
$$
=\E\big[\1_{\{X_1=1\}}\1_{\{S_{2}^{N-j}=k-\ell-1\}}\big] +
\E\big[(1-\1_{\{X_1=1\}})\1_{\{S_{2}^{N-j}=k-\ell\}}\big]                                              
$$
we obtain
\begin{eqnarray*}
\Phi_{k,j}(\ell) & = &
\E\big[\1_{\{X_1=1\}}\1_{\{S_{2}^{N-j}= k-\ell-1\}}\big] -\eps\E\big[\1_{\{S_{2}^{N-j}=k-\ell-1\}}\big]\\
 & & - \Big(\E\big[\1_{\{X_1=1\}}\1_{\{S_{2}^{N-j}=k-\ell\}}\big] -\eps\E\big[\1_{\{S_{2}^{N-j}=k-\ell\}}\big]\Big).
\end{eqnarray*}
We want an estimate for a term of the form 
\begin{equation}\label{unbout}
\E\big[\1_{\{X_1=1\}}\1_{\{S_{2}^{T}=t\}}\big] -\eps\E\big[\1_{\{S_{2}^{T}=t\}}\big],\;
0\leq t\leq T.
\end{equation}
We start by observing that 
\begin{eqnarray*}
\1_{\{S_{2}^{T}=t\}} & = & \1_{\{X_2=1\}}\1_{\{S_2^T=t\}}+\1_{\{X_2=0\}}\1_{\{S_2^T=t\}}
\\
& = & 
\1_{\{X_2=1\}}\1_{\{S_2^T=t\}}+\1_{\{X_2=0\}}\1_{\{S_3^T=t\}}
\\
& = &
\1_{\{X_2=1\}}\1_{\{S_2^T=t\}}+\big(1-\1_{\{X_2=1\}}\big)\1_{\{S_3^T=t\}},
\end{eqnarray*}
whence
$$
-\1_{\{X_2=1\}}
\leq
\1_{\{S_{2}^{T}=t\}}-\1_{\{S_3^T=t\}}
\leq
\1_{\{X_2=1\}}.
$$
More generally, we get for every $m\geq 1$
$$
-\1_{\{X_{m+1}=1\}}
\leq
\1_{\{S_{m+1}^{T}=t\}}-\1_{\{S_{m+2}^T=t\}}
\leq
\1_{\{X_{m+1}=1\}}.
$$
Summing these inequalities for $m=1,2,\ldots, p-1$ yields 
$$
\big|
\1_{\{S_{2}^{T}=t\}}-\1_{\{S_{p+1}^T=t\}}
\big|
\leq
\sum_{m=1}^{p-1} \1_{\{X_{m+1}=1\}}
$$
for every $p\geq 2$. Therefore we have the following bound for \eqref{unbout}:
$$
\big|\E\big[\1_{\{X_1=1\}}\1_{\{S_{2}^{T}=t\}}\big] -\eps\E\big[\1_{\{S_{2}^{T}=t\}}\big]\big|
\leq 
$$
$$
\big|
\E\big(\1_{\{X_1=1\}} \1_{\{S_{p+1}^T=t\}}\big) - \eps \E\big(\1_{\{S_{p+1}^T=t\}}\big)
\big|
+
\sum_{m=1}^{p-1} \E\left(\1_{\{X_1=1\}}\1_{\{X_{m+1}=1\}}\right)
+p\epsilon^{2}. 
$$
Collecting all the estimates we get for each $k$
$$
\sum_{j=0}^{N-p-1} \big|\Delta_k(j)\big|\le 2N \big[
R_{1}(\epsilon,N,p)+R_{2}(\epsilon,p)+2p\epsilon^{2}\big].
$$
For the last $p$ terms ($N-p\le j\le N-1$) in the sum \eqref{tele}, we
cannot use the above estimate. Instead we directly bound the terms to get
immediately
$$
\big|\Phi_{k,j}(\ell)\big|\le 4\epsilon
$$ 
whence
$$
\sum_{j=N-p}^{N-1} \big|\Delta_k(j)\big|\le 4p\epsilon\;.
$$
Therefore we obtain for each $k$
\begin{eqnarray}
\nonumber
\lefteqn{\big|\pee\big(S_1^N=k\big)-\pee\big(\tS_1^N=k\big) \big| \leq} \\ 
\label{deltaP}
&& \qquad 2N \big[R_{1}(\epsilon,N,p)+R_{2}(\epsilon,p)+2p\epsilon^{2}\big] +  4p\epsilon.
\end{eqnarray}
We now estimate the total variation norm between the law of
$S_1^N$ and that of $\tS_1^N$ which we write as
\begin{equation}\label{tvn}
\sum_{k=0}^{N-1}\big|\pee\big(S_1^N=k\big)-\pee\big(\tS_1^N=k\big)\big|=:I_{1}+I_{2}
\end{equation}
where 
$$
I_{2}=\sum_{k=M}^{N-1}\big|\pee\big(S_1^N=k\big)-\pee\big(\tS_1^N=k\big)\big|.
$$
We have at once
\begin{eqnarray*}
I_{2} & \le & \sum_{k=M}^{N-1}\pee\big(S_1^N=k\big)+ \sum_{k=M}^{N-1}\pee\big(\tS_1^N=k\big) \\
        &  = & 2\sum_{k=M}^{N-1}\pee\big(\tS_1^N=k\big)
                   +\sum_{k=M}^{N-1}\big[\pee\big(S_1^N=k\big)-\pee\big(\tS_1^N=k\big)\big]\\
        &  = & 2\sum_{k=M}^{N-1}\pee\big(\tS_1^N=k\big)+
\sum_{k=0}^{M-1}\big[\pee\big(\tS_1^N=k\big)-\pee\big(S_1^N=k\big)\big]\\
       & \le  & 2 \sum_{k=M}^{N-1}\pee\big(\tS_1^N=k\big)+ I_{1}.
\end{eqnarray*}
We now use the fact \cite{chenstein} that  for any $\lambda>0$ and any integer
$N\ge 1$, 
\begin{equation}\label{lecam}
\sum_{k=0}^\infty \Big|\pee(\tS_1^N=k) -
 \frac{e^{-\lambda}\lambda^k}{k!}\Big|\leq \frac{2\lambda^2}{N}\cdot
\end{equation}
and observe that
$$
\sum_{k=M}^{N-1}\pee\big(\tS_1^N=k\big)=\pee\big(\tS_1^N\ge M\big).
$$
Therefore, using \eqref{lecam} with $\lambda=N\epsilon$ we get
$$
\sum_{k=M}^{N-1}\pee\big(\tS_1^N=k\big)\le
2N\epsilon^{2}+e^{-N\epsilon}\frac{(N\varepsilon)^{M}}{M!} \cdot
$$
Hence
\begin{equation}\label{hideux}
I_2\leq 4N\epsilon^{2}+ 2e^{-N\epsilon}\frac{(N\varepsilon)^{M}}{M!} + I_{1}.
\end{equation}
On the other hand, we have from \eqref{deltaP} the obvious bound
\begin{equation}\label{hiun}
I_{1}\le 2M N \big[R_{1}(\epsilon,N,p)+R_{2}(\epsilon,p)+2p\epsilon^{2}\big]
+4Mp\epsilon.
\end{equation}
Using the triangle inequality, \eqref{tvn} and (\ref{lecam}) with  $\lambda=N\epsilon$, we obtain
\begin{align*}
d_{\mathrm{\scriptscriptstyle TV}} &
\big(X_1+\cdots+X_N,\textup{Poisson}(N\epsilon)\big)  \\
& \leq \frac12 \sum_{k=0}^{N-1}\big|\pee\big(S_1^N=k\big)-\pee\big(\tS_1^N=k\big)\big| 
+ \frac12 \sum_{k=0}^\infty \Big|\pee(\tS_1^N=k) -
 \frac{e^{-\lambda}\lambda^k}{k!}\Big|  \\
& \leq\frac12 \Big(I_{1}+I_{2}+2N\epsilon^{2}\Big)\cdot
\end{align*}
Using \eqref{hideux} and \eqref{hiun} we conclude the proof of the theorem.

\end{proof}

\section{A class of non-uniformly hyperbolic systems}\label{NUDS}

We work in the setting described in \cite{young1,young2} to which we refer
for more details. We first recall (most of) the axioms and then list
some of their consequences we use later on.

\subsection{Axioms}\label{axioms}

Let $T:M\circlearrowleft$ be a $C^2$ diffeomorphism of a finite-dimensional Riemannian manifold $M$. \newline
An embedded disk $\gamma\subset M$ is called an unstable disk if  for any $x,y\in\gamma$,
the distance $d(T^{-n}x,T^{-n}y)$ tends to $0$ exponentially fast as $n\to\infty$; it is called a stable
disk if for any $x,y\in\gamma$, the distance $d(T^{n}x,T^{n}y)$ tends to $0$ exponentially fast as $n\to\infty$.\newline
We say that a set $\Lambda$ has a hyperbolic product structure if there exist a continuous
family of unstable disks $\legu=\{\gamma^u\}$ and a continuous family of stable disks
$\legs=\{\gamma^s\}$ such that
\begin{enumerate}
\item $\dimu+\dims=n$ where $\dimu=\textup{dim}(\gamma^u)$, $\dims=\textup{dim}(\gamma^s)$
and $n=\textup{dim}(M)$;
\item the $\gamma^u$-disks are transversal to the $\gamma^s$-disks with the angles between them
bounded away from zero;
\item each $\gamma^u$-disk meets each $\gamma^s$-disk at exactly one point;
\item $\Lambda=(\cup\gamma^u)\cap (\cup\gamma^s)$.
\end{enumerate}
A central ingredient is a certain return-time function
$R:\Lambda\to\N$.
In the sequel, we denote by $\leb$ the Riemannian measure on $M$ and by $\leb_{\gamma}$ the measure
on $\legu$ induced by the restriction of the Riemannian  structure of $M$ to $\gamma$.

\begin{itemize}
\item[(P1)] There exists $\Lambda\subset M$ with a hyperbolic product structure
and such that $\leb_\gamma(\gamma\cap\Lambda)>0$ for every $\gamma\in\legu$.
\item[(P2)] There are pairwise disjoint sets $\Lambda_1,\Lambda_2,\ldots \subset \Lambda$
with the following properties:
\begin{itemize}
\item[(a)] Each $\Lambda_i$ has a hyperbolic product structure and its defining families
can be chosen to be $\legu$ and $\legs_i\subset \legs$; we call $\Lambda_i$ an $s$-subset;
similarly, one defines $u$-subsets.
\item[(b)] On each $\gamma^u$-disk, 
$\leb_\gamma\big((\Lambda\backslash\cup_i\Lambda_i)\cap \gamma\big)=0$ for every $\gamma\in\legu$.
\item[(c)] There exists $R_i\geq 0$ such that $T^{R_i}(\Lambda_i)$ is a $u$-subset
of $\Lambda$; moreover, for all $x\in\Lambda_i$ we require that 
$T^{R_i}(\gamma^s(x))\subset \gamma^s(T^{R_i}x)$ and
$T^{R_i}(\gamma^u(x))\supset \gamma^u(T^{R_i}x)$.
\item[(d)] For each $n$, there are at most finitely many $i$'s with $R_i=n$.
\item[(e)] $\min_i R_i\geq R_0$ for some $R_0>0$ depending only on $T$.
\end{itemize}
\end{itemize}
To state the remaining conditions we need to assume that there is a function
$s_0(x,y)$ (``separation time" of $x$ and $y$) which satisfies the following conditions
\begin{enumerate}
\item $s_0(x,y)\geq 0$ and it depends only on the $\gamma^s$-disks containing the two points;
\item the maximum number of orbits starting from $\Lambda$ that are pairwise separated before
time $n$ is finite for each $n$;
\item for $x,y\in\Lambda_i$, $s_0(x,y)\geq R_i + s_0\big(T^{R_i}(x),T^{R_i}(y)\big)$; in particular,
$s_0(x,y)\geq R_i$;
\item for $x\in\Lambda_i$, $y\in\Lambda_j$, $i\neq j$, we have $s_0(x,y)<R_i -1$.
\end{enumerate}
Let $T^u$ be the restriction of $T$ to $\gamma^u$. We assume that there exist $C>0$ and
$\alpha<1$ such that for all $x,y\in\Lambda$, the following conditions hold:
\begin{itemize}
\item[(P3)] Contraction along $\gamma^s$-disks: $d(T^n(x),T^n(y))\leq C \alpha^n$ for all $n\geq 0$
and $y\in\gamma^s(x)$.
\item[(P4)] Backward contraction and distortion along $\gamma^u$: for $y\in \gamma^u(x)$ and
$0\leq k\leq n<s_0(x,y)$, we have
\begin{itemize}
\item[(a)] $d(T^n(x),T^n(y))\leq  C \alpha^{s_0(x,y)-n}$;
\item[(b)]
$$
\log \prod_{i=k}^{n} \frac{\textup{det}\ DT^u(T^i(x))}{\textup{det}\ DT^u(T^i(y))}\leq C \alpha^{s_0(x,y)-n}.
$$
\end{itemize}
\item[(P5)] Convergence of $D(T^i \big|\gamma^u)$ and absolute continuity of $\legs$:
\begin{itemize}
\item[(a)] for $y\in\gamma^s(x)$ and $n\geq 0$,
$$
\log \prod_{i=n}^{\infty} \frac{\textup{det}\ DT^u(T^i(x))}{\textup{det}\ DT^u(T^i(y))}\leq C \alpha^{n};
$$
\item[(b)] for $\gamma,\gamma'\in \legu$, define $\Theta:\gamma\cap\Lambda \to \gamma' \cap\Lambda$
by $\Theta(x)=\gamma^s(x)\cap\gamma'$; then $\Theta$ is absolutely continuous and
$$
\frac{\d\big(\Theta^{-1}_* \leb_{\gamma'}\big)}{\d \leb_\gamma}(x)=
\prod_{i=0}^{\infty} \frac{\textup{det}\ DT^u(T^i(x))}{\textup{det}\ DT^u(T^i(\Theta(x)))}\cdot
$$
\end{itemize}
\end{itemize}
   
\subsection{Some properties}
   
As proved in \cite{young1}, if for some unstable manifold $\gamma\in\legu$, one has
\begin{equation}\label{srb1}
\sum_{p=1}^\infty \leb_\gamma\big\{x\in \gamma\cap \Lambda: R(x)>p\big\}
<\infty, 
\end{equation}
then $(M,T)$ admits an SRB measure which we denote by $\mu$.\newline
Define the set
$$
\attrac:=\bigcup_{i=1}^\infty \bigcup_{j=0}^{R_i -1} T^j(\Lambda_i).
$$  
This is the attractor of the system and its supports the SRB measure $\mu$.\newline
We recall that  for any measurable set $S$ we have the formula
\begin{equation}\label{mutour}
\mu(S)=\sum_{i=1}^{\infty}\sum_{j=0}^{R_{i}-1}m\big(T^{-j}\big(S\big)\cap
 \Lambda_{i}\big)
\end{equation}
where $m$ is the SRB measure for $(\Lambda,T^{R})$. We refer to \cite{young1} for
details.  The measure $m$ can be disintegrated using the foliation in local unstable manifolds. 
For any integrable function $g$ we have
\begin{equation}\label{desintm}
\int_{\Lambda} g\,\d m=
\int_{\legu} \d\nu(\gamma)\int_{\gamma}  g\, \d m_\gamma,
\end{equation}
where $\nu$ is the so-called transverse measure. Each measure
$m_\gamma$ has a density with respect to $\leb_\gamma$:
\begin{equation}\label{couic}
\d m_\gamma = \rho_\gamma\ \d\leb_\gamma,
\end{equation}
where
\begin{equation}\label{densite}
B^{-1}\le \rho_{\gamma}(x)\le B
\end{equation}
for some positive constant $B>1$ independent of 
$\gamma\in\legu$.\newline
Note that the measure $\leb_{\gamma}$ is not normalised. However, since
$m$ is a probability measure, we have
\begin{equation}\label{munormalise}
\int_{\legu} \d\nu(\gamma) \int_{\gamma}\rho_{\gamma}
\,\d\leb_{\gamma}(\gamma)=1\;.
\end{equation}
Given $\beta\in\ ]0,1]$, let $\mathcal{H}_\beta(M)$ be the Banach space of real-valued H\"older continuous
functions on $M$ ($\beta=1$ gives the Lipschitz functions). We denote by $\|\cdot\|_{\beta}$ the H\"older norm.
Using \cite{young1,young2} and Theorem B.1 in \cite{ccs},
we have the following decay of correlations for H\"older functions with respect to
the SRB measure $\mu$: there is a sequence $C(p)=C(p,\beta)$ of positive real numbers
tending to zero as $p\to\infty$, such that for any functions $\psi_{1},\psi_{2}\in \mathcal{H}_\beta(M)$,
we have
\begin{equation}\label{decor}
\left|\int \psi_{1}\cdot \psi_{2}\circ T^{p}\,\d\mu-\int \psi_{1}\,\d\mu
\int \psi_{2}\,\d\mu\ \right| \le C(p)\ \|\psi_{1}\|_{\beta}\ \|\psi_{2}\|_{\beta}.
\end{equation}
It was proved in \cite{young2} that
$$
C(p) = \Oun \ \sum_{k>p}  m\{R>k\}.
$$
Notice that \eqref{srb1} implies that
$$
\sum_{p=0}^\infty m\{R>p\}<\infty.
$$
The following positive function of $s\in\R$ will appear repeatedly:
$$
\Omega(s):= \sqrt{\sum_{i : R_i\geq s} R_i \ m(\Lambda_i)}.
$$
Notice that $\Omega(s)\to 0$ as $s\to+\infty$ and that
$$
C(p)=\Oun\ \Omega(p)^2.
$$
We will also use repeatedly the positive number  
\begin{equation}\label{leA}
A=\|DT\|_{L^{\infty}}+\|DT^{-1}\|_{L^{\infty}}+\|D^{2}T\|_{L^{\infty}}\;.
\end{equation}
Note that $A\ge 2$. 

\subsection{Poisson approximation}

We can now formulate precisely our main theorem which is loosely stated
in the introductory section:

\begin{theorem}\label{nuhpoisson}
Let $(M,T,\mu)$ be a dynamical system obeying the axioms of
Subsection \ref{axioms} where $\mu$ is the SRB measure.
Moreover assume that the return-time function $R$ has a exponential tail
(with respect to the measure $m$)
and that the local unstable manifolds are of dimension one.\newline
There exist positive constants $C,a,b$ such for any $r\in(0,1)$:
\begin{itemize}
\item There exists a set $\trestresmauvais_r$ such that 
$$
\mu(\trestresmauvais_r)\leq C r^b\, ;
$$
\item For all $x\notin \trestresmauvais_r$ and all $t>0$ one has
$$
d_{\mathrm{\scriptscriptstyle TV}}\big(Z_{r,x}(t), \textup{Poisson}(t)\big)\leq C \ r^a
$$
where 
$$
Z_{r,x}(t)=\sum_{j=0}^{\lfloor t/\mu(B_r(x))\rfloor} \1_{\scriptscriptstyle{B_r(x)}}\circ T^j 
$$
and $\textup{Poisson}(t)$ is a Poisson random variable of mean $t$.
\end{itemize}
\end{theorem}

\section{Proof of Theorem \ref{nuhpoisson}}\label{preuves}

We will apply Theorem \ref{abstract} to the class of non-uniformly hyperbolic dynamical systems described in
Section \ref{NUDS}.
We will take $X_n=\1_{B_r(x)}\circ T^{n-1}$, $n\geq 1$, where $B_r(x)$ denotes the ball of center $x$
and radius $r$, whence $\eps=\mu(B_r(x))$.
We will control the error terms $R_1(\epsilon,N,p)$ and $R_2(\epsilon,p)$ in Theorem \ref{abstract}.
From now on, we work under the assumptions of Theorem  \ref{nuhpoisson}.


\subsection{Estimation of $R_2(\epsilon,p)$}

We first estimate the measure of certain points $x$ coming back ``too quickly'' into
the ball $B_r(x)$.

\begin{lemma}\label{courtnonu}
Let
$$
\mauvais_{r}=\Big\{x\in \attrac\,\big|\, \exists\
1\le k\le \lfloor \mathfrak{c}\log (r^{-1})\rfloor,B_{r}(x)\cap T^{k}\big(B_{r}(x)\big)\neq\emptyset\Big\},
$$
where $\mathfrak{c}:=1/(6\log A)$. Then there exists a constant $D>0$ and, for any
$\mathfrak{a}\in\big(0,\frac{2}{3\log A}\big)$, there exists $\mathfrak{b}=\mathfrak{b}(\mathfrak{a})>0$
such for any $r\in (0,1)$
$$
\mu\big(\mauvais_{r}\big)\le 
D\ \log(r^{-1})\,\left[r^{\frac{\dimu \mathfrak{b}}{2}}+
\Omega^2\big(\mathfrak{a}\log (r^{-\frac12})\big)\right].
$$
\end{lemma}

\noindent Notice that this lemma holds for any $\dimu\geq 1$.

\medskip

\begin{proof}
Let $\mathfrak{a}_{0}>0$ such that $\mathfrak{a}_{0}<\mathfrak{c}$ to be chosen later on. 
We define the following sets:
\begin{eqnarray*}
\mauvais_{r}^{(1)} & := & 
\bigcup_{k=\lceil\mathfrak{a}_{0}\log (r^{-1})\rceil}^{\lfloor \mathfrak{c}\log (r^{-1})\rfloor}
\mathcal{N}_{r}(k),
\\ \\
\mauvais_{r}^{(2)} & := & 
\bigcup_{k=1}^{\lfloor\mathfrak{a}_{0}\log (r^{-1})\rfloor}
\mathcal{N}_{r}(k),
\end{eqnarray*}
where
$$
\mathcal{N}_{r}(k):=
\Big\{x\in \attrac\big| B_{r}(x)\cap T^{k}\big(B_{r}(x)\big)
\neq\emptyset\Big\}.
$$
By definition we have
$$
\mauvais_{r}= \mauvais_{r}^{(1)} \cup 
 \mauvais_{r}^{(2)}.
$$
We now derive a uniform estimate of
$\mu\big(\mathcal{N}_{r}(k)\big)$ for $k\ge \lfloor\mathfrak{a}_{0}\log (r^{-1})\rfloor$.\newline
Assume there exist $\gamma\in \legu$ and integers $i,j$ such that $T^{k-j}(x)\in \gamma\cap \Lambda_{i}$.
Let $z\in \mathcal{N}_{r}(k)$ be such that $T^{k-j}(z)\in \gamma\cap
\Lambda_{i}$. 
Note that by the Markov property, 
$T^{k+R_{i}-j}(z)\in \gamma\big(T^{k+R_{i}-j}(x)\big)$.
We will use the notations
$\hat x=T^{R_{i}-j}(x)$, and $\hat z=T^{R_{i}-j}(z)$.\newline
We distinguish two cases. 
Assume first 
$$
d\big(T^{k}(\hat x),T^{k}(\hat z)\big)\le 2 d(\hat x,\hat z)\;.
$$
From (P4)(a) in Section \ref{NUDS} and since $k\ge \lfloor\mathfrak{a}_{0}\log
r^{-1}\rfloor$, 
we have (since $T^{k}(\hat z)\in
\gamma^u(T^{k}(\hat x)$) 
$$
d(\hat x,\hat z)\le C\alpha^{k}\le C\alpha^{\mathfrak{a}_{0}\log (r^{-1})}.
$$
Hence
$$
d\big(T^{k}(\hat x),T^{k}(\hat z)\big)
\le 2 d(\hat x,\hat z)
\le 2 C\alpha^{\mathfrak{a}_{0}\log (r^{-1})}\;.
$$
We now consider the case 
$$
d\big(T^{k}(\hat x),T^{k}(\hat z)\big)\geq 2 d(\hat x,\hat z)\;.
$$
We observe that  $B_{r}(x)\cap T^{k}\big(B_{r}(x)\big)\neq\emptyset$
implies that there exists $y\in B_{r}(x)$ such that $T^{k}(y)\in
B_{r}(x)$. Therefore
$$
d\big(x,T^{k}(x)\big)\le d\big(x,T^{k}(y)\big)+
d\big(T^{k}(x),T^{k}(y)\big)\le r+A^{k}d(x,y)\le \big(A^{k}+1\big)r\;.
$$
Let $\mathfrak{a}\in \big( 0,\frac{2}{3\log A}\big)$ and assume that  $R_{i}\le \mathfrak{a}\log (r^{-1})$.
Then
$$
2 d(\hat x,\hat z) \le 
d\big(T^{k}(\hat x),T^{k}(\hat z))
\le  d\big(T^{k}(\hat x),\hat x)+d(\hat x,\hat z)+
 d\big(\hat z,T^{k}(\hat z)\big)
$$
$$
\le d(\hat x,\hat z) +2 \,A^{R_{i}-j}\big(A^{k}+1\big)\,r
$$
$$
\le d(\hat x,\hat z) +4 \, A^{(\mathfrak{c}+\mathfrak{a})\log (r^{-1})}r\;.
$$
It follows that 
$$
d(\hat x,\hat z) \le 4\, A^{(\mathfrak{c}+\mathfrak{a})\log (r^{-1})}r\;.
$$
This implies
$$
d\big(T^{k}(\hat x),T^{k}(\hat z)\big) 
\le A^{k} \ d(\hat x,\hat z) 
\le 4 \ A^{(2\mathfrak{c}+\mathfrak{a})\log (r^{-1})}r = 4\ r^{1-(2\mathfrak{c}+\mathfrak{a})\log A}\ .
$$
Since $\mathfrak{c}=1/(6\log A)$ and $\mathfrak{a}\in \ (0,\frac{2}{3\log A})$, we have $1-(2\mathfrak{c}+\mathfrak{a})\log A>0$.
Combining both cases, we obtain
$$
d\big(T^{k}(\hat x),T^{k}(\hat z)\big)\le C' r^{\mathfrak{b}}
$$
where $C'$ is a positive constant (independent of $x$, $z$ and $r$) and
$$
\mathfrak{b}=\min\big\{-\mathfrak{a}_{0}\log\alpha\,,\, \frac23 -\mathfrak{a}\log A\big\}\;.
$$
Letting $\gamma'=T^{R_{i}}\big(\Lambda_{i}\cap\gamma\big)\in\legu$,
it follows immediately that for any $i$ we have
$$
\leb_{\gamma'}\big\{T^{R_{i}}\big(\Lambda_{i}\cap\gamma\cap 
T^{k-j}\big(\mathcal{N}_{r}(k)\big)\big)\big\}\le C'' r^{\dimu\,\mathfrak{b}}\;,
$$
for a positive constant $C''$ independent of $\gamma$, $i$, $k$.
Using (P4)(b) and (P5)(b) we get 
\begin{equation}\label{lapp}
\frac{\leb_{\gamma}\big\{\Lambda_{i}\cap T^{k-j}\big(\mathcal{N}_{r}(k)\big)\big\}}{\leb_{\gamma}\big(\Lambda_{i})}\le 
C''' r^{\dimu\,\mathfrak{b}}\;.
\end{equation}
From \eqref{mutour} and the invariance of the measure, we have
$$
\mu\big(\mathcal{N}_{r}(k)\big)
=\mu\big(T^{k}\big(\mathcal{N}_{r}(k)\big)\big)
$$
$$
\le
\sum_{i,\,R_{i}\le \mathfrak{a}\log (r^{-1})}\sum_{j=0}^{R_{i}-1}m\big\{T^{-j}
\big(T^{k}\big(\mathcal{N}_{r}(k)\big)\big)\cap\Lambda_{i}\big\}+
\Omega^2\big(\mathfrak{a}\log (r^{-1})\big)\,.
$$
For fixed $k$, $i$ and $j$ we can use the expression \eqref{desintm} to
obtain
$$
m\big\{T^{k-j}\big(\mathcal{N}_{r}(k)\big)\cap\Lambda_{i}\big\}=\int_{\legu}\d\nu(\gamma)
\int_{\gamma\cap \Lambda_{i}\cap 
T^{k-j}\big(\mathcal{N}_{r}(k)\big)}\rho_{\gamma}\;\d\leb_{\gamma}\;.
$$
Using the estimate \eqref{lapp} we get
$$
m\big\{T^{k-j}\big(\mathcal{N}_{r}(k)\big)\cap\Lambda_{i}\big\}\le 
C''' r^{\dimu\mathfrak{b}}\; \int_{\legu}\d\nu(\gamma)
\int_{\gamma\cap \Lambda_{i} }\rho_{\gamma}\;\d\leb_{\gamma}
$$
$$
=C''' r^{\dimu\mathfrak{b}}\; m\big(\Lambda_{i}\big)\;.
$$
This implies
\begin{eqnarray}
\nonumber
\mu\big(\mathcal{N}_{r}(k)\big) & \le & C''' r^{\dimu\mathfrak{b}}
\sum_{i,\,R_{i}\le \mathfrak{a}\log(r^{-1})}\sum_{j=0}^{R_{i}-1}
m\big(\Lambda_{i}\big)
+
\Omega^2\big(\mathfrak{a}\log (r^{-1})\big)
\\
\label{zestime}
& \le & C''' r^{\dimu\,\mathfrak{b}}+\Omega^2\big(\mathfrak{a}\log(r^{-1})\big)
\;.
\end{eqnarray}
This yields 
$$
\mu\big(\mauvais_{r}^{(1)}\big)\le 
\mathfrak{c} \ \log (r^{-1}) \Big[C''' r^{\dimu\,\mathfrak{b}}+\Omega^2\big(\mathfrak{a}\log(r^{-1})\big)\Big]\;.
$$
We now consider the case $1\le k<\lfloor\mathfrak{a}_{0}\log (r^{-1})\rfloor$ to estimate
$\mu\big(\mauvais_{r}^{(2)}\big)$. 
For a such a $k$, we define an integer $p(k)$ by
$$
p(k)=\big\lfloor \log_{2}(\mathfrak{a}_{0}\log (r^{-1}))-\log_{2}(k)\big\rfloor+1\;.
$$
and a radius
$$
r'(k)=2^{p(k)}\;\frac{A^{k2^{p(k)}}-1}{A^{k}-1}\,r\;.
$$
Observe that
$$
\mathfrak{a}_{0}\log {r'(k)}^{-1}\leq \mathfrak{a}_{0}\log {r}^{-1}\leq k':=k2^{p(k)}\le 4 \mathfrak{a}_{0}\log {r}^{-1}
$$
and since $A\ge 2$
$$
r'(k)\leq A^{2k 2^{p(k)}} r \le A^{8\mathfrak{a}_{0}\log {r}^{-1}} r \,.
$$
Applying Lemma \ref{itervide} we get
$$
\mathcal{N}_{r}(k)\subset \mathcal{N}_{r'(k)}(k')\subset \mathcal{N}_{A^{8\mathfrak{a}_{0}\log {r}^{-1}} r}(k')\, .
$$
Using the estimate \eqref{zestime} and choosing
$$
\mathfrak{a}_0=\frac{1}{16\log(A)}
$$
we obtain
$$
\mu\big(\mauvais_{r}^{(2)}\big)
= 
\mu\left(\bigcup_{k=1}^{\lfloor\mathfrak{a}_{0}\log (r^{-1})\rfloor}\mathcal{N}_{r}(k)\right)
\leq 
\mu\left(\bigcup_{\ell=\lfloor\mathfrak{a}_{0}\log (r^{-1})\rfloor}^{2\lfloor\mathfrak{a}_{0}\log (r^{-1})\rfloor}
\mathcal{N}_{A^{8\mathfrak{a}_{0}\log {r}^{-1}} r}(\ell)\right)
\le
$$
$$
\lfloor\mathfrak{a}_{0}\log (r^{-1})\rfloor \,\left(
C''' \, r^{\frac{\dimu\,\mathfrak{b}}{2}}+
\Omega^2\big(\mathfrak{a}\log r^{-\frac12}\big) \right)\;.
$$
The result follows by putting together all the estimates.
\end{proof}

\bigskip

In the next proposition,  we provide an estimate for the error term $R_2(\mu(B_r(x)),p)$.

\bigskip

\begin{proposition}\label{nimportequoi}
There exist constants $C>0$ and $\mathfrak{s}>0$ such that 
for any $r\in (0,1)$, for any $\mathfrak{a}\in(0,\frac{2}{3\log A})$, for $\mathfrak{b}=\mathfrak{b}(\mathfrak{a})$
as in Lemma \ref{courtnonu}, and there exists a measurable set $\mathcal{U}_r$ satisfying
$$
\mu(\mathcal{U}_r) \leq 
$$
$$
C\left[\Omega(\mathfrak{s}\log(r^{-1})/3)+ r^{\mathfrak{s}}+
\log (r^{-1})\ r^{\frac{\dimu\mathfrak{b}}{2}}+\log (r^{-1})\ \Omega^{2}\big(\mathfrak{a}
\log(r^{-1/2})\big) \right].
$$
such that for any $x\in \attrac\backslash\mathcal{U}_r$ and for all $p\ge2$,
\begin{align*}
R_{2} & \big(\mu\big(B_{r}(x)\big),p\big) \\
& \le C\mu\big(B_{r}(x)\big)\bigg[
(\log (r^{-1}))^{3}\ \Omega\left(\frac{1}{3}\;\min\left\{\frac14,\;
\frac{-\mathfrak{c}\log\alpha}{4}\right\}\log (r^{-1})\right) \\
 & \qquad\qquad+\max\left\{r^{\frac12},\alpha^{\frac{\mathfrak{c}}{2}\log (r^{-1})}\right\}
+ p\big( r^{\frac{\mathfrak{r}}{2}} +r^{-3-n}\; \Omega^2\big((\log (r^{-1}))^{2}\big)\big)\bigg].
\end{align*}
The constants $\mathfrak{r}$ and $\mathfrak{c}$ are those appearing in Lemma \ref{mesureboule}
and Lemma \ref{courtnonu}, respectively.
\end{proposition}

\medskip

\begin{proof}
We have
$$
R_2\big(\mu(B_r(x)),p\big)=
$$
$$
\left(\sum_{\ell=1}^{\lfloor \mathfrak{c}\log (r^{-1})\rfloor-1}
+
\sum_{\ell=\lfloor \mathfrak{c}\log (r^{-1})\rfloor}^{p-1}\right)
\E\bigg(\1_{B_r(x)} \1_{B_r(x)}\circ T^{\ell}\bigg).
$$
The first sum is controlled using Lemma \ref{courtnonu}: it is empty
if $x\in \attrac\backslash\mauvais_{r}$.
Thus, from now on, we assume that $\ell \ge \lfloor \mathfrak{c}\log (r^{-1})\rfloor$.\newline
Let
$$
s:=\frac{1}{3}\;\min\left\{\frac14,\;
\frac{-\mathfrak{c}\log\alpha}{4}\right\}\log (r^{-1})
$$
and
$$
\ell_{0}:=(\log (r^{-1}))^{2}.
$$
We use Corollary \ref{besico2mu} with $q=s$ and $\omega=\omega_1$ where
$$
\omega_1=
\sqrt{\sum_{i, R_i > s} \sum_{j=0}^{R_i -1} m\big( \Lambda_i\big)}
$$
and formula \eqref{mutour} to get
$$
\E\Big(\1_{B_r(x)} \ \1_{B_r(x)}\circ T^{\ell}\Big) \leq 
$$
$$
\sum_{i, R_i \leq s} \sum_{j=0}^{R_i -1} m\big\{\Lambda_i \cap T^{-j} B_r(x)\cap T^{-j-\ell} B_r(x)\big\}
+ \omega_1 \ \mu(B_r(x)),
$$
for any $x\in\attrac$ outside of the set $\mathscr{C}_{\omega_1}$ such that
\begin{equation}\label{NY1}
\mu(\mathscr{C}_{\omega_1}) \leq p(n)\ \omega_1.
\end{equation}
For each $i$ such that $R_i\leq s$, we define the set
$$
\tilde{\Lambda}_i=\big\{ x\in \Lambda_i : \forall j\leq \ell_0, R((T^R)^{j}(x))\leq s\big\}.
$$
Let $\lambda_0,\lambda_1$ be the finite positive measures defined by
\begin{eqnarray*}
\lambda_0(S) & = & \sum_{i,R_i\leq s} \sum_{j=0}^{R_i -1} m(\Lambda_i\cap T^{-j} S)\\
\lambda_1(S) & = & \sum_{i,R_i\leq s} \sum_{j=0}^{R_i -1} m(\Lambda_i\backslash \tilde{\Lambda}_i \cap T^{-j} S).
\end{eqnarray*}
We have 
\begin{eqnarray*}
\lambda_1(M) & = &\sum_{i,R_i\leq s} \sum_{j=0}^{R_i -1} m(\Lambda_i\backslash \tilde{\Lambda}_i) \leq (s+1) \ \sum_{i,R_i\leq s} 
m(\Lambda_i\backslash \tilde{\Lambda}_i)\\
& \leq &
(s+1)\ m\left( \bigcup_{j=0}^{\ell_0} \big(T^R\big)^{-j} \{R\geq s\}\right)\\
& \leq & (s+1)\ \ell_0 \ m\{R\geq s\}\, ,
\end{eqnarray*}
where the last inequality follows from the $T^R$-invariance of $m$. We now apply Lemma \ref{besicodeux}
to the measures $\lambda_0$ and $\lambda_1$ defined above, and $\omega=\omega_2$ defined as
$$
\omega_2=\sqrt{(s+1)\ \ell_0 \ m\{R\geq s\}}.
$$
We have
\begin{align}
\nonumber
& \sum_{i, R_i \leq s} \sum_{j=0}^{R_i -1} m\big( \Lambda_i \cap T^{-j} B_r(x)\cap T^{-j-\ell} B_r(x)\big) \\
\nonumber
& \leq 
\sum_{i, R_i \leq s} \sum_{j=0}^{R_i -1} m\big( \tilde{\Lambda}_i \cap T^{-j} B_r(x)\cap T^{-j-\ell} B_r(x)\big) +
\omega_2\ \mu_0(B_r(x))\\
& 
\label{lemachin}
\leq \sum_{i, R_i \leq s} \sum_{j=0}^{R_i -1} m\big( \tilde{\Lambda}_i \cap T^{-j} B_r(x)\cap T^{-j-\ell} B_r(x)\big) +
\omega_2\ \mu(B_r(x))
\end{align}
for any $x\in\attrac$ outside of the set $\mathscr{C}_{\omega_2}(\lambda_0,\lambda_1,r)$ such that
$$
\lambda_0(\mathscr{C}_{\omega_2}(\lambda_0,\lambda_1,r)) \leq p(n)\ \omega_2
$$
which implies 
\begin{equation}\label{NY2}
\mu(\mathscr{C}_{\omega_2}(\lambda_0,\lambda_1,r)) \leq p(n)\ \omega_2 + \omega_1^2.
\end{equation}
For any $\gamma\in\legu$ and any finite sequence of integers $i_0,\ldots,i_{m}$ ($m\geq 1$),
we define the following (non-empty) subset of $\gamma$:
$$
\zeta_{i_0,\ldots,i_{m}}(\gamma)=\{x\in \gamma\cap \tilde\Lambda_{i_0} : (T^R)^p(x)\in \Lambda_{i_p}\;\forall 1\leq p \leq m\}.
$$
For any integers $i_0$, $j<R_{i_0},\ell$ and $\gamma\in\legu$, for any $r>0$,
we define 
$$
\mi_{\gamma,i_0,j,\ell,r}=\big\{ (i_0,\ldots,i_{m})\;\textup{minimal}\;\textup{such that}\; 
|T^j \zeta_{i_0,\ldots,i_{m}}(\gamma)|\leq r\;\textup{and}\;
$$
$$
\sum_{k=0}^m R_{i_k}(x) \geq j+\ell\; \textup{for}\;x\in \zeta_{i_0,\ldots,i_{m}}(\gamma)\big\},
$$
where $|\cdot |$ denotes the diameter of $\zeta_{i_0,\ldots,i_{m}}(\gamma)$.

By `$(i_0,\ldots,i_{m})$ minimal' we mean that for the sequence $(i_0,\ldots,i_{m-1})$
one of the two conditions is violated.
Observe that from minimality we have
either
$$
\sum_{k=0}^{m-1} R_{i_k}(x) < j+\ell
$$
or
$$
\sum_{k=0}^{m-1} R_{i_k}(x) \geq  j+\ell\quad\textup{and}\quad |T^j \zeta_{i_0,\ldots,i_{m-1}}(\gamma)|>r.
$$
It is easy to verify that for any $\gamma,i_0,j<R_{i_0},\ell,r$, $\mi_{\gamma,i_0,j,\ell,r}$ is a (finite)
partition of $\gamma\cap \tilde{\Lambda}_{i_0}$ up to a set of Lebesgue measure
zero.

\begin{itemize}

\item
If $\sum_{k=0}^{m-1} R_{i_k}(x) < j+\ell$, $(T^R)^m \zeta_{i_0,\ldots,i_{m}}(\gamma)=\Lambda \cap \gamma_1$ for
some $\gamma_1\in\legu$.
Since $\ell<R_{i_m}\leq s$, we have for some constant $c>0$
\begin{equation}\label{zulu1}
|T^{j+\ell}\zeta_{i_0,\ldots,i_{m}}(\gamma)| \geq c\ A^{-s}.
\end{equation}
\item 
If $\sum_{k=0}^{m-1} R_{i_k}(x) \geq j+\ell$ and $|T^j \zeta_{i_0,\ldots,i_{m-1}}(\gamma)|>r$,
then we have for some $\gamma_2\in\legu$
$$
(T^R)^{m-1} \zeta_{i_0,\ldots,i_{m-1}}(\gamma)=\Lambda \cap \gamma_2
$$
and
$$
(T^R)^{m-1} \zeta_{i_0,\ldots,i_{m}}(\gamma)=\Lambda_{i_m} \cap \gamma_2.
$$
Since the maximal expansion factor is $A$ and $R_{i_m}\leq s$, we have
$$
|\Lambda_{i_m} \cap \gamma_2| \geq A^{-s}.
$$
Hence
$$
\frac{|(T^R)^{m-1} \zeta_{i_0,\ldots,i_{m}}(\gamma)|}{|(T^R)^{m-1} \zeta_{i_0,\ldots,i_{m-1}}(\gamma)|}
\geq A^{-s}.
$$
If $\dimu=1$, the distortion of the differential along a backward orbit of a local unstable
manifold is uniformly bounded. Therefore, since $j\leq \sum_{k=0}^{m-1} R_{i_k}(x)$, 
we get
$$
\frac{|T^{j} \zeta_{i_0,\ldots,i_{m}}(\gamma)|}{|T^{j}\zeta_{i_0,\ldots,i_{m-1}}(\gamma)|}
\geq C\ A^{-s}
$$
which implies
$$
|T^{j} \zeta_{i_0,\ldots,i_{m}}(\gamma)| \geq C\ r \ A^{-s}.
$$
By the uniform backward contraction along unstable manifolds (cf. (P4)(a) in Section \ref{NUDS}), 
and since $\ell \ge \lfloor \mathfrak{c}\log (r^{-1})\rfloor$, we get
\begin{equation}\label{zulu2}
|T^{j+\ell} \zeta_{i_0,\ldots,i_{m}}(\gamma)| \geq C\ \alpha^{-\ell}\ r \ A^{-s}\geq C\ \alpha^{-\mathfrak{c}\log (r^{-1})}\ r \ A^{-s}.
\end{equation}
\end{itemize}
We now estimate the first term in \eqref{lemachin}. We will use the fact that,
if $\tau\in \mi_{\gamma,i,j,\ell,r}$,
$T^j \zeta_{\tau}(\gamma) \cap B_r(x)\neq \emptyset$
and $|T^j \zeta_{\tau}(\gamma) |\leq r$, then
we have $T^j \zeta_{\tau}(\gamma) \subset B_{2r}(x)$.
We have
\begin{align}
\label{santiago}
& \sum_{i, R_i \leq s} \sum_{j=0}^{R_i -1} m\big( \tilde{\Lambda}_i \cap T^{-j} B_r(x)\cap T^{-j-\ell} B_r(x)\big)\\
\nonumber
& =
\sum_{j=0}^{s-1}  m\left(\bigcup_{i: j+1\leq R_i\leq s}\tilde{\Lambda}_i \cap T^{-j} B_r(x)\cap T^{-j-\ell} B_r(x)\right)\\
\nonumber
& =
\sum_{j=0}^{s-1}  \int_{\legu} \d\nu(\gamma)\int_{\gamma\cap \bigcup_{i: j+1\leq R_i\leq s}\tilde{\Lambda}_i }  
\1_{\{T^{-j} B_r(x)\}}\ \1_{\{T^{-j-\ell} B_r(x)\}}
\rho_{\gamma} \d \leb_\gamma\\
\nonumber
& =
\sum_{j=0}^{s-1}  \int_{\legu} \d\nu(\gamma)
\sum_{i: j+1\leq R_i\leq s} \,  \sum_{\tau \in \mi_{\gamma,i,j,\ell,r}} \\
\nonumber
& \qquad\qquad \qquad \int_{\zeta_{\tau}(\gamma)} 
\1_{\{T^{-j} B_r(x)\}}\ \1_{\{T^{-j-\ell} B_r(x)\}}\rho_{\gamma} \d \leb_\gamma\\
\nonumber
& \leq 
\sum_{j=0}^{s-1}  \int_{\legu} \d\nu(\gamma)
\sum_{i: j+1\leq R_i\leq s} \,\sum_{\stackrel{\tau \in \mi_{\gamma,i,j,\ell,r}}{ \zeta_{\tau}(\gamma)
\cap T^{-j} B_r(x)\neq \emptyset}} \\
\nonumber
& \qquad\qquad\qquad
\int_{\zeta_{\tau}(\gamma)}  \1_{\{T^{-j} B_{2r}(x)\}}\ \1_{\{T^{-j-\ell} B_r(x)\}}
\rho_{\gamma} \d \leb_\gamma\\
\nonumber
& =
\sum_{j=0}^{s-1}  \int_{\legu} \d\nu(\gamma)
\sum_{i: j+1\leq R_i\leq s} \,\sum_{\stackrel{\tau \in
    \mi_{\gamma,i,j,\ell,r}}{ \zeta_{\tau}(\gamma)
\cap T^{-j} B_r(x)\neq \emptyset}}\\
\nonumber
& \qquad
\frac{ \int_{\zeta_{\tau}(\gamma)} \1_{\{T^{-j} B_{2r}(x)\}}\ \1_{\{T^{-j-\ell} B_r(x)\}}
\rho_{\gamma} \d \leb_\gamma}{\int_{\zeta_{\tau}(\gamma)}  \1_{\{T^{-j} B_{2r}(x)\}}
\rho_{\gamma} \d \leb_\gamma}\; 
\int_{\zeta_{\tau}(\gamma)}  \1_{\{T^{-j} B_{2r}(x)\}}
\rho_{\gamma} \d \leb_\gamma.
\end{align}
We bound the prefactor of the previous integral as follows:
\begin{align*}
\frac{ \int_{\zeta_{\tau}(\gamma)} \1_{\{T^{-j} B_{2r}(x)\}}\ \1_{\{T^{-j-\ell} B_r(x)\}}
\rho_{\gamma} \d \leb_\gamma}{\int_{\zeta_{\tau}(\gamma)}  \1_{\{T^{-j} B_{2r}(x)\}}
\rho_{\gamma} \d \leb_\gamma} &=
\frac{ \int_{\zeta_{\tau}(\gamma)}  \1_{\{T^{-j-\ell} B_r(x)\}}
\rho_{\gamma} \d \leb_\gamma}{\int_{\zeta_{\tau}(\gamma)}  
\rho_{\gamma} \d \leb_\gamma}\\
& \leq
C\ \frac{\int_{\zeta_{\tau}(\gamma)}   \1_{\{T^{-j-\ell} B_r(x)\}}
\d \leb_\gamma}{ \int_{\zeta_{\tau}(\gamma)}  \d \leb_\gamma}\cdot
\end{align*}
Let $m'\le m$ be the smallest integer such that
$$
\sum_{k=0}^{m'} R_{i_k}(x)\ge j+\ell\;.
$$
Let
$$
t:= \sum_{k=0}^{m'} R_{i_k}(x)-j-\ell\;.
$$
We have for $\tau\in \mi_{\gamma,i,j,\ell,r}$
$$
T^{j+\ell+t} (\zeta_{\tau}(\gamma)) \subset
\Lambda\cap \tilde{\gamma},\quad \textup{for some}\;\;\tilde{\gamma}\in\legu.
$$
From (P4)(b) (Section \ref{NUDS}) we obtain since  $0\le t\le s$ 
\begin{align*}
\frac{ \int_{\zeta_{\tau}(\gamma)}  \1_{\{T^{-j-\ell} B_r(x)\}}
\d \leb_\gamma}{ \int_{\zeta_{\tau}(\gamma)}  \d \leb_\gamma} & \leq
C\ \frac{\leb_{\tilde{\gamma}}\big(T^t(B_r(x))\cap \tilde{\gamma}\big)}
{\leb_{\tilde{\gamma}}\big(T^{j+\ell+t} (\zeta_{\tau}(\gamma))
\cap \tilde{\gamma}\big)}\\
& \le C A^{2s} r \;\frac{1}{\big|T^{j+\ell} (\zeta_{\tau}(\gamma)) \big|}\\
& \le C A^{3s} \max\left\{r,\alpha^{\mathfrak{c}\log (r^{-1})}
\right\},
\end{align*}
where the last inequality follows from \eqref{zulu1} and \eqref{zulu2}.
Therefore we have using \eqref{santiago} and the above estimates 
\begin{align*}
& \sum_{i, R_i \leq s} \sum_{j=0}^{R_i -1} m\big( \tilde{\Lambda}_i \cap
T^{-j} B_r(x)\cap T^{-j-\ell} B_r(x)\big)\\
& \le
C A^{3s} \max\left\{r,\alpha^{\mathfrak{c}\log (r^{-1})}\right\} \times \\
& \qquad \sum_{j=0}^{s-1}  \int_{\legu} \d\nu(\gamma)
\sum_{i: j+1\leq R_i\leq s} \
\sum_{\stackrel{\tau \in\mi_{\gamma,i,j,\ell,r}}{ \zeta_{\tau}(\gamma)
\cap T^{-j} B_r(x)\neq \emptyset}}
 \int_{\zeta_{\tau}(\gamma)} \1_{\{T^{-j} B_{2r}(x)\}}
\rho_{\gamma} \d \leb_\gamma\\
& \leq
CA^{3s} \max\left\{r,\alpha^{\mathfrak{c}\log (r^{-1})}\right\}\;
\sum_{j=0}^{s-1}  
\sum_{i: j+1\leq R_i\leq s}\int_{\legu} \d\tilde{\nu}(\gamma)
 \int_{\gamma} \1_{\{T^{-j} B_{2r}(x)\}}
\rho_{\gamma} \d \leb_\gamma\\
& \leq 
C  A^{3s} \max\left\{r,\alpha^{\mathfrak{c}\log (r^{-1})}\right\}\; \mu(B_{2r}(x)),
\end{align*}
where the last inequality follows from \eqref{mutour}. 
Using Lemma \ref{besico1} we get for $x\notin \mathscr{E}_{r,\mathfrak{s}}$ that
\eqref{santiago} is bounded from above by
$$
C\   A^{3s} \ r^{-\mathfrak{s}}\ \max\left\{r,\alpha^{\mathfrak{c}\log(r^{-1})}\right\} \;  \mu(B_{r}(x)).
$$
Collecting the above estimates, we obtain for any $\ell\leq \ell_0$ that
\begin{align*}
&\E\Big[\1_{B_r(x)} \ \1_{B_r(x)}\circ T^{\ell}\Big]\leq\\
& \qquad \qquad \left( \omega_1+\omega_2+ C\   A^{3s} \ r^{-\mathfrak{s}}\
\max\left\{r,\alpha^{\mathfrak{c}\log (r^{-1})}\right\}\right)\ \mu(B_r(x))
\end{align*}
for any $x$ outside the set
\begin{equation}\label{leuun}
\mathcal{T}_r:=\mathscr{C}_{\omega_1}\cup \mathscr{C}_{\omega_2} \cup \mathscr{E}_{r,\mathfrak{s}}.
\end{equation}
\bigskip

We now consider the case $\ell>\ell_0$.
We define the following Lipschitz function:
$$
\psi_{x,r}(y)=\left\{\begin{array}{lcl}
1 & \mathrm{if} & d(x,y)\le r\\
2-\frac{d(x,y)}{r} & \mathrm{if} & r\le d(x,y)\le 2r\\
0 & \mathrm{if} & 2r\le d(x,y)\;.
\end{array}\right.
$$
The Lipschitz constant of $\psi_{x,r}$ is $1/r$. We have
$$
\E\Big(\1_{B_r(x)} \ \1_{B_r(x)}\circ T^{\ell}
\Big)
\le 
\int \psi_{x,r}(y) \; \psi_{x,r}(T^{\ell}(y))\;\d\mu(y).
$$
Using the decay of correlations 
\eqref{decor}, we obtain  for
any $x$, for any $r\in(0,1)$ and for any integer $\ell$ 
$$
\int \psi_{x,r}(y) \; \psi_{x,r}(T^{\ell}(y))\;\d\mu(y)\le
\left(\int \psi_{x,r}(y) \;\d\mu(y)\right)^{2}+r^{-2} \; C(\ell).
$$
Since $\psi_{x,r}\le \1_{B_{2r}(x)}$, using Lemma \ref{besico1} and Lemma \ref{mesureboule}, we
get for $x\notin \mathscr{E}_{r,\mathfrak{s}}\cup \mathcal{J}_r$ we get 
\begin{align*}
\E\Big(\1_{B_r(x)} \ \1_{B_r(x)}\circ T^{\ell}\Big) & \le \mu\big(B_{2r}(x)\big)^{2}+ r^{-2}\;C(\ell)\\
& \le r^{-2\mathfrak{s}}  \mu\big(B_{r}(x)\big)^{2}+r^{-2} \; C(\ell)\\
& \le C\ r^{\mathfrak{r}/2}\mu\big(B_{r}(x)\big)+r^{-2} \; C(\ell)\, ,
\end{align*}
where in the last inequality we chose $\mathfrak{s}\leq \mathfrak{r}/4$.
Using Lemma \ref{mesbou} for $\mathfrak{g}=1$,
we can write for $x\notin \mathcal{A}_{r}\cup \mathcal{J}_r$ 
$$
\E\Big[\1_{B_r(x)} \ \1_{B_r(x)}\circ T^{\ell}\Big]
\le \tilde{C}\big[r^{\mathfrak{r}/2} +r^{-3-n} \; C(\ell)\big]\mu\big(B_{r}(x)\big)
$$
for a constant $\tilde{C}>0$. 

\bigskip

We now fix 
$$
\mathfrak{s}=\min\left\{\frac14,\;\frac{1}{4\log A},\;
\frac{-\mathfrak{c}\log\alpha}{4},\;\frac{\mathfrak{r}}{4}\right\}
$$
and define the set 
$$
\mathcal{U}_r:= \mathcal{T}_r \cup \mathcal{A}_{r}\cup \mathcal{J}_r\cup \mauvais_{r}.
$$
Using \eqref{NY1}, \eqref{NY2}, \eqref{leuun}, Lemma \ref{besico1}, Lemma \ref{mesureboule}, Lemma \ref{mesbou} and Lemma \ref{courtnonu}
we obtain
\begin{align*}
\mu(\mathcal{U}_r) & \leq 
\mu(\mathscr{C}_{\omega_1})+\mu(\mathscr{C}_{\omega_2})+\mu(\mathscr{E}_{r,\mathfrak{s}}) + \mu\big(\mathcal{J}_r\big)+
\mu( \mathcal{A}_{r})+\mu(\mauvais_{r})\\
& \leq C\left[
\omega_1 + \omega_1^{2}+\ \omega_2 + r^{\mathfrak{s}}+r+\Omega(\mathfrak{s}\log(r^{-1}))  \right. \\
& \qquad \left. +\log (r^{-1})\,\left(r^{\frac{\dimu\,\mathfrak{b}}{2}}+
\Omega^2\big(\mathfrak{a}\log (r^{-\frac12})\big)
\right)
\right]\\
&  \le  C\left[ \Omega(s)+ r^{\mathfrak{s}}+\Omega\big(\mathfrak{s}\log(r^{-1})\big) \right. \\
& \qquad \left. +\log (r^{-1})\,r^{\frac{\dimu\,\mathfrak{b}}{2}}
+(\log (r^{-1}))\Omega^2\big(\mathfrak{a}\log(r^{-1/2})\big) \right]\\
& \leq C\left[ \Omega\big(\mathfrak{s}\log(r^{-1})/3\big)+ r^{\mathfrak{s}} \right. \\
& \qquad \left. +\log (r^{-1})\,r^{\frac{\dimu\,\mathfrak{b}}{2}}+
\log (r^{-1})\,\Omega^2\big(\mathfrak{a}\log(r^{-1/2})\big)\right],
\end{align*}
since $\Omega$ is a decreasing function.
We obtain
$$
\sum_{\ell=\lfloor \mathfrak{c}\log (r^{-1})\rfloor}^{p-1}
\E\bigg(\1_{B_r(x)} \ \1_{B_r(x)}\circ T^{\ell}\bigg)
$$
$$
\leq
\left\{
\begin{array}{lcl}
\ell_0 \left( \sqrt{\ell_0} \Omega(s) + C  A^{3s}
\max\left\{r^{3/4},\alpha^{3(\mathfrak{c}\log (r^{-1}))/4}\right\} \right)\mu\big(B_{r}(x)\big)
& \textup{if} & p\leq \ell_0\\ \\
\ell_0 \left[\sqrt{\ell_0} \Omega(s) + C  A^{3s}\max\Big\{r^{3/4},\alpha^{3(\mathfrak{c}\log (r^{-1}))/4}\Big\}\right] \mu\big(B_{r}(x)\big)\\ \\
+ p \left[r^{\mathfrak{r}/2} + r^{-3-n}\; \Omega^2(\ell_0)\right]\mu\big(B_{r}(x)\big)
& \textup{if} & p>\ell_0
\end{array}\right.
$$
for any $x\notin \mathcal{U}_r$. Recall that 
$$
s=\frac{1}{3}\;\min\left\{\frac14,\;
\frac{-\mathfrak{c}\log\alpha}{4}\right\}\times \log (r^{-1})
$$
and
$$
\ell_{0}=(\log (r^{-1}))^{2}.
$$
We get
\begin{align*}
& \sum_{\ell=\lfloor \mathfrak{c}\log (r^{-1})\rfloor}^{p-1}
\E\bigg(\1_{B_r(x)} \ \1_{B_r(x)}\circ T^{\ell}\bigg)
\le\\
&   \qquad  \qquad C\mu\big(B_{r}(x)\big)\bigg[
(\log (r^{-1}))^{3}\ \Omega\left(\frac{1}{3}\;\min\left\{\frac14,\;
\frac{-\mathfrak{c}\log\alpha}{4}\right\}\log (r^{-1})\right) +\\
& \qquad \qquad \qquad\max\left\{r^{1/2},\alpha^{\mathfrak{c}\log (r^{-1})/2}\right\}+
p\big( r^{\mathfrak{r}/2} +r^{-3-n}\; \Omega^2\big((\log (r^{-1}))^2\big)\bigg].
\end{align*}
This ends the proof.
\end{proof}

\subsection{Estimation of $R_1(\epsilon,N,p)$}

We shall have to deal with the measure
of certain coronas:
For any $r\in(0,1]$, $x\in\attrac$ and any $\delta>1$
we define the corona $\cour$ by 
$$
\cour=B_{r}(x)\backslash B_{r-r^{\delta}}(x).
$$
Let 
\begin{equation}\label{defv}
v:=1+\left\lceil -\frac{\log A}{\log \alpha}\right\rceil\cdot
\end{equation}
Define the set $\hat{\Lambda}_{\mathfrak{q},r}$ as the set of points $x\in \Lambda$
such that:
$$
R\big((T^R)^\ell(x)\big)\leq \mathfrak{q} \log (r^{-1})
$$ 
whenever $\ell$ is such that:
$$
\sum_{q=0}^{\ell-1} R\big((T^R)^q(x)\big)<(v+1)\mathfrak{q}\log (r^{-1}).
$$
For $x\in \hat{\Lambda}_{\mathfrak{q},r}$, define
$$
L_{\mathfrak{q},r}(x)=\min\bigg\{\ell \ \Big| \ \sum_{q=0}^{\ell} R\big((T^R)^q(x)\big)\geq (v+1)\mathfrak{q}\log (r^{-1})\bigg\}.
$$
Observe that
$$
(v+1)\mathfrak{q}\log (r^{-1})\leq \sum_{q=0}^{L_{\mathfrak{q},r}(x)} R\big((T^R)^q(x)\big)\leq (v+2)\mathfrak{q}\log (r^{-1}).
$$
Define the following set of pieces of unstable disks
$$
\mathcal{G}_{\mathfrak{q},r}=
\left\{ \big(T^R\big)^{-L_{\mathfrak{q},r}(x)}\Big(\gamma^u\Big(
\big(T^R\big)^{L_{\mathfrak{q},r}(x)}(x)\Big)\Big)
\cap \Lambda, \, \forall x\in  \hat{\Lambda}_{\mathfrak{q},r}\right\}\,.
$$
Observe that $\mathcal{G}_{\mathfrak{q},r}$ is a partition of $\hat{\Lambda}_{\mathfrak{q},r}$ and that the
function $x\mapsto L_{\mathfrak{q},r}(x)$ is constant on the elements of $\mathcal{G}_{\mathfrak{q},r}$.

\bigskip

\begin{lemma}\label{geometrique}
There exists a constant $C>0$ such that for any $\mathfrak{q}>0$, for any $r\in(0,1)$
and for any $\eta\in \mathcal{G}_{\mathfrak{q},r}$ and for any $j\leq \mathfrak{q} \log (r^{-1})$, we
have, for all $\delta>1$,
$$
m_\gamma\big\{T^{-j}\big(\cour\big)\cap  \eta\big\} \leq C\ r^{\delta/2}\ A^{\mathfrak{q} \log (r^{-1})}
$$
where $\gamma$ is the element of $\legu$ containing $\eta$.
\end{lemma}
\begin{proof}
Since $T$ is a diffeomorphism we have
$$
m_\gamma\big\{T^{-j}\big(\cour\big)\cap  \eta\big\}= m_\gamma\big\{T^{-j}\big(\cour\cap T^j(\eta)\big)\big\}.
$$
We can write for any $y\in\eta$
$$
T^j(\eta)= T^{j-R(y)}\big(T^{R(y)}(\eta)\big).
$$
Observe that from the definition of $L_{\mathfrak{q},r}(y)$ above that for all $y\in\eta$
$$
T^{L_{\mathfrak{q},r}(y)}(\eta)=\gamma'\cap \Lambda
$$
for some $\gamma'\in\legu$. Therefore, from (P4)(a) and the definition of $v$ in \eqref{defv}, for
all $y\in\eta$, we have
$$
\big| T^{R(y)}(\eta)\big| \leq \alpha^{L_{\mathfrak{q},r}(y)-R(y)} \leq \alpha^{v \mathfrak{q}\log(r^{-1})}
\leq A^{-\mathfrak{q}\log(r^{-1})}\ r^{-\mathfrak{q}\log\alpha}.
$$
It follows that $T^{R(y)}(\eta)\subset \gamma"\in\legu$. Hence $T^{R(y)}(\eta)$ is a small embedded
disk. From the above estimate on $\big| T^{R(y)}(\eta)\big|$ we deduce that, for any $0\leq j \leq R(y)$,
$T^j(\eta)$ is an embedded disk and there is a control on the size and on the embedding which
is uniform in $r$. 
Namely, 
Since $T^j(\eta)$ is almost flat, there is a uniform constant $C>0$ such that
$$
\big| \cour\cap T^j(\eta)\big| \leq C\ r^{\delta/2}.
$$
The lemma follows from \eqref{leA} and the fact that $0\leq j \leq \mathfrak{q}\log(r^{-1})$.
\end{proof}

\begin{proposition}\label{courons}
There exist constants $C>0$, $r_{0}\in(0,1)$, such that 
for any $r\in(0,r_{0})$ and any $\mathfrak{q}>0$, there exists a measurable set
$\tresmauvais_{r}$ satisfying
$$
\mu\big(\tresmauvais_{r}\big)\le C r
$$
and such that for any $x\in\attrac\backslash \tresmauvais_{r}$
we have for all $\delta>1$
$$
\mu\big(\cour\big)
\le C\ \mu\big(B_{r}(x)\big)\times
$$
$$
\left[r^{\frac{\delta}{2}-n-1}\ A^{(v+3)\mathfrak{q} \log (r^{-1})}+r^{-n-1}(1+v\mathfrak{q} \log (r^{-1}))^2\
 \Omega^2\big(\mathfrak{q}\log(r^{-1})\big)\right]\, ,
$$
where $v$ is defined in \eqref{defv}.
\end{proposition}

\begin{proof}
We define
$$
\tresmauvais_{r}=\left\{x\,\big|\, \mu\big(B_{r}(x)\big)\le r^{n+1}\right\}\;.
$$
It follows from Lemma \ref{mesbou} that
$$
\mu\big(\tresmauvais_{r}\big)\le C\,r\;.
$$
We have 
$$
\mu\big(\cour\big) \leq 
\sum_{i,R_{i}< \lfloor\mathfrak{q}\log (r^{-1})\rfloor}^{\infty}
\sum_{j=0}^{R_{i}-1}m\big\{T^{-j}\big(\cour\big)\cap\Lambda_{i}\big\}
 +\Omega^2\big(\mathfrak{q}\log(r^{-1})\big).
$$
Define the sets $\hat\Lambda_i=\Lambda\cap \hat\Lambda_{\mathfrak{q},r}$, where $\hat\Lambda_{\mathfrak{q},r}$
is defined above.\newline
Now observe that from the definition of $\hat\Lambda_{\mathfrak{q},r}$ we have
\begin{align*}
& \sum_{i,R_{i}< \lfloor\mathfrak{q}\log (r^{-1})\rfloor}^{\infty}
\sum_{j=0}^{R_{i}-1}m\big(\Lambda_{i}\backslash \hat\Lambda_i\big)\\
& \qquad\leq
\mathfrak{q} \log (r^{-1})\; 
\sum_{i,R_{i}< \lfloor\mathfrak{q}\log (r^{-1})\rfloor}^{\infty} m\big(\Lambda_{i}\backslash \hat\Lambda_i\big)\\
& \qquad\leq
\mathfrak{q} \log (r^{-1})\; m\left( \bigcup_{q=0}^{\lfloor(v+2)\mathfrak{q} \log (r^{-1})\rfloor} \big(T^R\big)^{-q}\{R>\mathfrak{q} \log (r^{-1})\}\right).
\end{align*}
Using the $T^R$-invariance of $m$ we get
\begin{equation}\label{sinus}
\mu\big(\cour\big) \leq 
\end{equation}
$$
\sum_{i,R_{i}< \lfloor\mathfrak{q}\log (r^{-1})\rfloor}^{\infty}
\sum_{j=0}^{R_{i}-1}m\big\{T^{-j}\big(\cour\big)\cap\hat\Lambda_{i}\big\}
 +(1+v\mathfrak{q} \log (r^{-1}))^2\ \Omega^2\big(\mathfrak{q}\log(r^{-1})\big).
$$
For any $j<R_i<\mathfrak{q} \log (r^{-1})$, and $\gamma\in\legu$, we have
\begin{align*}
& \frac{m_\gamma\big\{T^{-j}\big(\cour\big)\cap\hat\Lambda_{i}\big\}}{m_\gamma\big(\hat\Lambda_{i}\big)}\\
& =
\frac{\sum_{\eta\in\mathcal{G}_{\mathfrak{q},r}} m_\gamma\big\{T^{-j}\big(\cour\big)\cap\hat\Lambda_{i}\cap \eta\big\}}
{\sum_{\eta\in\mathcal{G}_{\mathfrak{q},r}} m_\gamma\big(\hat\Lambda_{i}\cap \eta\big)}\\
& \leq
\sup_{\eta\in\mathcal{G}_{\mathfrak{q},r}, \eta\subset \gamma\cap \hat\Lambda_{i}}  
\frac{m_\gamma\big\{T^{-j}\big(\cour\big)\cap\hat\Lambda_{i}\cap \eta\big\}}
{m_\gamma\big(\hat\Lambda_{i}\cap \eta\big)}\cdot
\end{align*}
Observe that from the definition of $L_{\mathfrak{q},r}(x)$ above that for all $x\in\eta$
$$
T^{L_{\mathfrak{q},r}(x)}(\eta)=\gamma'\cap \Lambda
$$
for some $\gamma'\in\legu$. 
If $\dimu=1$ then 
$$
|\eta| \geq A^{-(v+2)\mathfrak{q}\log (r^{-1})}
$$
by (P4)(b) and \eqref{leA}.
Therefore by using Lemma \ref{geometrique} we obtain
$$
\frac{m_\gamma\big\{T^{-j}\big(\cour\big)\cap\hat\Lambda_{i}\big\}}{m_\gamma\big(\hat\Lambda_{i}\big)}
\leq
C\ r^{\delta/2}\ A^{(v+3)\mathfrak{q} \log (r^{-1})}.
$$
Using \eqref{desintm} and the previous inequality, we have 
\begin{align*}
\frac{m\big\{T^{-j}\big(\cour\big)\cap\hat\Lambda_{i}\big\}}{m\big(\hat\Lambda_{i}\big)} &=
\frac{\int_{\legu}d\nu(\gamma)\ m_\gamma\big\{T^{-j}\big(\cour\big)\cap\hat\Lambda_{i}
\big\}}{\int_{\legu} d\nu(\gamma)\ m_\gamma\big(\hat\Lambda_{i}\big)}\\
& \leq C\ r^{\delta/2}\ A^{(v+3)\mathfrak{q} \log (r^{-1})}\;.
\end{align*}
This implies, using \eqref{sinus} and \eqref{mutour}, that 
\begin{align*}
&\mu\big(\cour\big) \\
&\leq C r^{\delta/2}A^{(v+3)\mathfrak{q} \log (r^{-1})}\times \\
 & \quad\sum_{i,R_{i}< \lfloor\mathfrak{q}\log (r^{-1})\rfloor}^{\infty}
\sum_{j=0}^{R_{i}-1}
m\big(\hat\Lambda_{i}\big) 
+(1+v\mathfrak{q} \log (r^{-1}))^2 \Omega^2\big(\mathfrak{q}\log(r^{-1})\big)\\
& 
\le C r^{\delta/2} A^{(v+3)\mathfrak{q} \log (r^{-1})}\
 +(1+v\mathfrak{q} \log (r^{-1}))^2\ \Omega^2\big(\mathfrak{q}\log(r^{-1})\big).
\end{align*}
The proposition follows since $x\notin\tresmauvais_{r}$.
\end{proof}

\begin{proposition}\label{decornuh}
There exist constants $C>0$ and $\mathfrak{s}>0$ such that 
for any $r\in (0,1)$, for any $\mathfrak{a}\in(0,\frac{2}{3\log A})$, for
$\mathfrak{b}=\mathfrak{b}(\mathfrak{a})$
as in Lemma \ref{courtnonu}, and for  any $\mathfrak{p}_{0}>0$ and $\mathfrak{p}>0$, 
there exists a measurable subset $\trestresmauvais_{r}$ of the attractor $\attrac$ satisfying
\begin{align*}
& \mu(\trestresmauvais_{r}) \leq 
C\left[\Omega\big(\mathfrak{s}\log(r^{-1})/3\big)+ r^{\mathfrak{s}}+\log (r^{-1})\ r^{\frac{\dimu\mathfrak{b}}{2}}\right.\\
& \quad\qquad\;\qquad\left.+\log (r^{-1})\ \Omega\big(\mathfrak{a}
\log(r^{-1/2})\big)^{2} +r+
\Omega^2\big(\mathfrak{p}\log (r^{-1})\big)\right].
\end{align*}
such that for any $x\in \attrac\backslash\trestresmauvais_{r}$, 
we have for any integers $p$,
$\ell$ and $0\le q\le \ell$
\begin{align*}
& \left|\E\bigg(\1_{\{X_{1}=1\}}
\1_{\{S_{p+1}^{p+1+\ell}=q\}}\bigg)-\mu\big(B_{r}(x)\big) 
 \E\bigg( \1_{\{S_{p+1}^{p+1+\ell}=q\}}\bigg)\right|\le \\
&  C\mu\big(B_{r}(x)\big)
\bigg[r + r^{-n-1}(1+v\mathfrak{p}_{0} \log (r^{-1}))^2\
 \Omega^2\big(\mathfrak{p}_{0}\log(r^{-1})\big)\\
& +\ell r^{-n-1}(1+v\mathfrak{p}\log (r^{-1}))^2\
 \Omega^2\big(\mathfrak{p}\log(r^{-1})\big)\\
& +C\;r^{-3\big(n+2+(v+3)\mathfrak{p}_{0}\log A\big)}\;
\Omega^{2}\left(2\big(n+2+(v+3)\mathfrak{p}\log A\big)\;\frac{\log r}{\log\alpha}\right)\\
& +(\log (r^{-1}))^{3}\ \Omega\left(\frac{1}{3}\;\min\left\{\frac14,\;
\frac{-\log\alpha}{24\log A}\right\}\log (r^{-1})\right)\\
&+\max\left\{r^{\frac12},\alpha^{\frac{\log (r^{-1})}{12\log A}}\right\}\\
&+ 4\big(n+2+(v+3)\mathfrak{p}\log A\big)\;\frac{\log r}{\log\alpha}\;
\Big( r^{\frac{\mathfrak{r}}{2}} +r^{-3-n}\; \Omega^2\big((\log
r^{-1})^{2}\big)\Big)\bigg]\,,
\end{align*}
where $v$ is defined in \eqref{defv}.
\end{proposition}

\begin{proof}
Let $\tresmauvais_{r}$ be as in Proposition \ref{courons}. 
From now on we assume that $x\in\attrac\backslash \tresmauvais_{r}$.\newline
Let $\delta_{0}>1$. Define the function $\phi_{x,r}$ by
$$
\phi_{x,r}(y)=\1_{B_{r-r^{\delta_{0}}}(x)}(y)+\frac{r-d(x,y)}{r^{\delta_{0}}}\left(
\1_{B_{r}(x)}(y)-\1_{B_{r-r^{\delta_{0}}}(x)}(y)\right)\;.
$$
It is left to the reader to verify that this function is Lipschitz with
a Lipschitz constant  $r^{-\delta_{0}}$ (uniform in $x$). 
It follows easily using
Proposition \ref{courons}  with $\delta=\delta_{0}$ to be chosen later on
and $\mathfrak{q}=\mathfrak{p}_{0}$, that 
\begin{align*}
0\le & \E\Big(\1_{\{X_{1}=1\}}\1_{\{S_{p+1}^{p+1+\ell}=q\}}\Big)
-\E\Big(\phi_{x,r}\1_{\{S_{p+1}^{p+1+\ell}=q\}}\Big)\\
& \quad \le \E\bigg(\1_{B_{r}(x)}
\1_{\{S_{p+1}^{p+1+\ell}=q\}}\bigg)-\E\bigg(\1_{B_{r-r^{\delta_{0}}}(x)}
\1_{\{S_{p+1}^{p+1+\ell}=q\}}\bigg) \\
& \quad \le \E\bigg(\left(\1_{B_{r}(x)}
-\1_{B_{r-r^{\delta_{0}}}(x)}\right)
\1_{\{S_{p+1}^{p+1+\ell}=q\}}\bigg) \\
& \quad \le \mu\big(B_{r}(x)\big)-\mu\big(B_{r-r^{\delta_{0}}}(x)\big)\\
&\quad = \mu\big(\mathscr{C}_{r,\delta_{0}}\big) \\
& \quad \le C\ \left[r^{\frac{\delta_{0}}{2}-n-1}\ A^{(v+3)\mathfrak{p}_{0} \log (r^{-1})}+ \right. \\ 
& \qquad\qquad \left. r^{-n-1}(1+v\mathfrak{p}_{0} \log (r^{-1}))^2\
\Omega^2\big(\mathfrak{p}_{0}\log(r^{-1})\big)\right]\mu\big(B_{r}(x)\big).
\end{align*}
We now estimate the term $\E\Big(\phi_{x,r}\1_{\{S_{p+1}^{p+1+\ell}=q\}}\Big)$
using the decay of correlations. Let $p'=[p/2]$, and let (see Lemma
\ref{lev})
$$
\mathscr{Y}_{p',\ell}(x,r)=\bigcup_{k=p'}^{p'+\ell}\mathscr{V}_{k}(x,r)\;.
$$
From the definition of the sets $\mathscr{V}_{k}(x,r)$, the function 
$$
\psi=\1_{\{S_{p'}^{p'+1+\ell}=q\}}\;\1_{\mathscr{Y}^{c}_{p',\ell}(x,r)}
$$
is $L^{\infty}$ and constant on stable manifolds. We would like
to use the decay of correlations proved in \cite{young1,young2}.
Unfortunately, the function $\psi$ is not H\"older continuous.
However, it is known that for $\psi$ constant on local stable
manifolds, the proof works as well and leads to an estimate
where the H\"older norm of $\psi$ is replaced by its $L^\infty$
norm. This follows easily from the observation that, in this
case, Approximation \#1 in  \cite[Section 4.1]{young1} is
not necessary. The rest of the proof is identical.
This yields the estimate
\begin{multline*}
\left|
\E\bigg(\phi_{x,r}\,\left(\1_{\{S_{p'}^{p'+1+\ell}=q\}}\;
\1_{\mathscr{Y}^{c}_{p',\ell}(x,r)}\right)\circ T^{p+1-p'}\bigg) \right. \\
\left. -\E\big(\phi_{x,r}\big)\E\bigg(\1_{\{S_{p'}^{p'+1+\ell}=q\}}\;
\1_{\mathscr{Y}^{c}_{p',\ell}(x,r)}\bigg)\right|
\end{multline*}
$$
\le 
C\;r^{-\delta_{0}}\;\Omega^{2}(p/2)\;.
$$
From Lemma \ref{lev}, we have
$$
\E\left[\phi_{x,r}\left(\1_{\{S_{p'+1}^{p'+1+\ell}=q\}}
 \1_{\mathscr{Y}_{p',\ell}(x,r)}\right)\circ T^{p-p'+1} \right]
\le \sum_{k=p'}^{p'+\ell}\mu\big(\mathscr{V}_{k}(x,r)\big)\le 
$$
$$
  \le \sum_{k=p'}^{p'+\ell}\;
\mu\big(\tilde{\mathscr{C}}_{r,k\log\alpha/\log r}\big)\;.
$$
If $\alpha^{p'}<r/2$, we have by using Proposition \ref{courons}
with $\mathfrak{q}=\mathfrak{p}$ and suitable $\delta$'s,
\begin{align*}
& \left|
\E\bigg(\1_{\{S_{p'}^{p'+1+\ell}=q\}}\;
\1_{\mathscr{Y}^{c}_{p',\ell}(x,r)}\bigg)
-\E\bigg(\1_{\{S_{p'}^{p'+1+\ell}=q\}}\bigg)
\right|\\
& \le \mu\big(\mathscr{Y}_{p',\ell}(x,r)\big)\\
& \le \sum_{k=p'}^{p'+\ell}\;\mu\big(\tilde{\mathscr{C}}_{r,k\log\alpha/\log r}\big)\\
& \le \sum_{k=p'}^{p'+\ell}\;\left(
\mu\big(\mathscr{C}_{r+\alpha^{k},k\log\alpha/\log(r+\alpha^{k})}\big)
+
\,\mu\big(\mathscr{C}_{r,k\log\alpha/\log r}\big)\right)\\
& \le C \left[r^{-n-1}\;\alpha^{p'/2}\ A^{(v+3)\mathfrak{p} \log (r^{-1})} \right. \\
&  \quad\qquad \left. +\ell r^{-n-1}(1+v\mathfrak{p} \log (r^{-1}))^2\
\Omega^2\big(\mathfrak{p}\log(r^{-1})\big)\right]\mu\big(B_{r}(x)\big).
\end{align*}
Using again Proposition \ref{courons} with
$\mathfrak{q}=\mathfrak{p}_{0}$ and
$$
\delta=\delta_{0}=2\big(n+2+(v+3)\mathfrak{p}_{0}\log A\big),
$$ 
we get the estimate  
\begin{align*}
0 & \le  \mu\big(B_{r}(x)\big)-\int \phi_{x,r}\;\d\mu \\
& \le \mu\big(B_{r}(x)\big)-\mu\big(B_{r-r^{\delta_{0}}}(x)\big)\\
& = \mu\big(\mathscr{C}_{r,\delta_{0}}\big)\\
& \leq C\ \left[r+r^{-n-1}(1+v\mathfrak{p}_{0} \log (r^{-1}))^2\ \Omega^2\big(\mathfrak{p}_{0}\log(r^{-1})\big)\right]\mu\big(B_{r}(x)\big)\;.
\end{align*}
If 
$$
p>p_{*}=
4\big(n+2+(v+3)\mathfrak{p}\log A\big)\;\frac{\log r}{\log\alpha}\;,
$$
we conclude that 
\begin{align*}
& \left|\E\bigg(\1_{\{X_{1}=1\}}
\1_{\{S_{p+1}^{p+1+\ell}=q\}}\bigg)-\mu\big(B_{r}(x)\big) 
 \E\bigg( \1_{\{S_{p+1}^{p+1+\ell}=q\}}\bigg)\right|\leq \\
&
C \bigg[r + r^{-n-1}(1+v\mathfrak{p}_{0} \log (r^{-1}))^2\ \Omega^2\big(\mathfrak{p}_{0}\log(r^{-1})\big)\\
&
\quad +\ell r^{-n-1}(1+v\mathfrak{p}\log (r^{-1}))^2\ \Omega^2\big(\mathfrak{p}\log(r^{-1})\big)\\
&
\quad +C\;r^{-3\big(n+2+(v+3)\mathfrak{p}_{0}\log A\big)}\;
\Omega^{2}\Big(2\big(n+2+(v+3)\mathfrak{p}\log A\big)\;\frac{\log r}{\log\alpha}\Big)
\bigg] \ \mu\big(B_{r}(x)\big)\;.
\end{align*}
The proposition follows in the case $p>p_{*}$. 

\noindent We now consider the case
$p\le p_{*}$. 
We can write
\begin{align*}
& \E\bigg(\1_{\{X_{1}=1\}}
\1_{\{S_{p+1}^{p+1+\ell}=q\}}\bigg)\\
& \quad =
\E\bigg(\1_{\{X_{1}=1\}}
\1_{\{S_{p+1}^{p+1+\ell}=q\}}\prod_{j=p}^{p_{*}}\1_{\{X_{j}=0\}}\bigg)\\
&
\qquad +\E\left(\1_{\{X_{1}=1\}}
\1_{\{S_{p+1}^{p+1+\ell}=q\}}\left(1-\prod_{j=p}^{p_{*}}\1_{\{X_{j}=0\}}\right)\right)\\
&
\quad=\E\left(\1_{\{X_{1}=1\}}
\1_{\{S_{p_{*}+1}^{p_{*}+1+\ell}=q\}}\prod_{j=p}^{p_{*}}\1_{\{X_{j}=0\}}\right)\\
&
\qquad+\E\left(\1_{\{X_{1}=1\}}
\1_{\{S_{p+1}^{p+1+\ell}=q\}}\left(1-\prod_{j=p}^{p_{*}}\1_{\{X_{j}=0\}}\right)\right)\\
&
\quad =\E\bigg(\1_{\{X_{1}=1\}}
\1_{\{S_{p_{*}+1}^{p_{*}+1+\ell}=q\}}\bigg)\\
&
\qquad -\E\left(\1_{\{X_{1}=1\}}
\1_{\{S_{p_{*}+1}^{p_{*}+1+\ell}=q\}}
\left(1-\prod_{j=p}^{p_{*}}\1_{\{X_{j}=0\}}\right)\right)\\
&\qquad +\E\left(\1_{\{X_{1}=1\}}
\1_{\{S_{p+1}^{p+1+\ell}=q\}}
\left(1-\prod_{j=p}^{p_{*}}\1_{\{X_{j}=0\}}\right)\right)\;.
\end{align*}
Therefore, using the invariance of the measure $\mu$ and the inequality
$$
1-\prod_{j=p}^{p_{*}}\1_{\{X_{j}=0\}} \leq \sum_{j=p}^{p_{*}}\1_{\{X_{j}=1\}}\,,
$$
we obtain
\begin{align*}
&\left|\E\bigg(\1_{\{X_{1}=1\}}
\1_{\{S_{p+1}^{p+1+\ell}=q\}}\bigg)-\mu\big(B_{r}(x)\big) \
 \E\bigg( \1_{\{S_{p+1}^{p+1+\ell}=q\}}\bigg)\right|\\
&
\qquad \le \left|\E\bigg(\1_{\{X_{1}=1\}}
\1_{\{S_{p_{*}+1}^{p_{*}+1+\ell}=q\}}\bigg)-\mu\big(B_{r}(x)\big) \
 \E\bigg( \1_{\{S_{p_{*}+1}^{p_{*}+1+\ell}=q\}}\bigg)\right|\\
&
\qquad\qquad +\ 2\sum_{j=p}^{p_{*}} \E\bigg(\1_{\{X_{1}=1\}}
\1_{\{X_{j}=1\}}\bigg)\,.
\end{align*}
The first term is estimated as before, and the second term is bounded by
 $R_{2}\big(\mu\big(B_{r}(x)\big),p_{*}\big)$, which is estimated using
 proposition \ref{nimportequoi} for $x\in\attrac\backslash
 \mathcal{U}_{r}$. The proposition follows if we take
$$
\trestresmauvais_{r}=\tresmauvais_{r}\cup\mathcal{U}_{r}\;. 
$$ 
\end{proof}

\subsection{End of proof }

\begin{proposition}
There exist constants $C>0$ and $\mathfrak{s}>0$ such that 
for any $r\in (0,\frac12[$, for any $\mathfrak{a}\in(0,\frac{2}{3\log A})$, for
  $\mathfrak{b}=\mathfrak{b}(\mathfrak{a})$
as in Lemma \ref{courtnonu}, and for  any $\mathfrak{p}_{0}>0$ and $\mathfrak{p}>0$, 
there exists a measurable subset $\trestresmauvais_{r}$ of the attractor $\attrac$
containing $\tresmauvais_{r}$ satisfying
$$
\mu(\trestresmauvais_{r}) \leq 
$$
$$
C\left[\Omega(\mathfrak{s}\log(r^{-1})/3)+ r^{\mathfrak{s}}+
\log (r^{-1})\ r^{\frac{\dimu\,\mathfrak{b}}{2}} \right. 
$$
$$
\qquad\qquad \qquad \qquad \left. +\log (r^{-1})\ 
\Omega^{2}\big(\mathfrak{a}\log(r^{-1/2})\big) +r+
\Omega^2(\mathfrak{p}\log (r^{-1}))\right].
$$
such that for any $x\in \attrac\backslash\trestresmauvais_{r}$, 
we have for any integers $p$, $N$ and $M$, 
the  error term in Theorem \ref{abstract} is bounded by
\begin{align*}
& R\big(\mu\big(B_{r}(x)\big),N,p,M\big) \le C\;\bigg[
NM\;\mu\big(B_{r}(x)\big)\times\\
& 
\bigg(
r + r^{-n-1}(1+v\mathfrak{p}_{0} \log (r^{-1}))^2\
 \Omega^2\big(\mathfrak{p}_{0}\log(r^{-1})\big)\\
&
\quad +N\, r^{-n-1}(1+v\mathfrak{p}\log (r^{-1}))^2\
\Omega^2\big(\mathfrak{p}\log(r^{-1})\big)\\
&
\quad+C\;r^{-3\big(n+2+(v+3)\mathfrak{p}_{0}\log A\big)}\;
\Omega^{2}\Big(2\big(n+2+(v+3)\mathfrak{p}\log A\big)\;\frac{\log r}{\log\alpha}\Big)\\
&
\quad+(\log (r^{-1}))^{3}\ \Omega\left(\frac{1}{3}\;\min\left\{\frac14,\;
\frac{-\log\alpha}{24\log A}\right\}\log (r^{-1})\right)
\\
&
\quad +\max\left\{r^{\frac12},\alpha^{\frac{\log (r^{-1})}{12\log A}}\right\}\\
&\quad + 4\big(p+n+2+(v+3)\mathfrak{p}\log A\big)\;\frac{\log r}{\log\alpha}\; \big( r^{\frac{\mathfrak{r}}{2}}
+r^{-3-n}\; \Omega^2\big((\log (r^{-1}))^{2}\big)\big)
\bigg)\\
&
+M\,p\, \mu\big(B_{r}(x)\big)\,\big(1+N\mu\big(B_{r}(x)\big)\big)\\
& + 
\frac{\big(\mu\big(B_{r}(x)\big) N\big)^{M}}{M!}\ e^{-\mu\big(B_{r}(x)\big) N}
+ N\mu\big(B_{r}(x)\big)^{2}
\bigg],
\end{align*}
for any $M<N$ and $p<N$. 
\end{proposition}
\begin{proof}
This result follows at once from Propositions \ref{nimportequoi} and 
\ref{decornuh}  with $\ell \le N$.
\end{proof}

We now finish the proof of Theorem \ref{nuhpoisson}.  Let  $x\in
\attrac\backslash\trestresmauvais_{r}$. 
We  choose for a fixed real number $t>0$
$$
N=[t/\mu\big(B_{r}(x)\big)]\;.
$$
Since $x\notin \trestresmauvais_{r}$, we have
$\mu\big(B_{r}(x)\big)>r^{n+1}$. 
We  choose $p=\Oun\log(r^{-1})$ and $M=1+[\log(r^{-1})]$.
If there are two constants $C>0$ and $\theta>0$ such that for any $s>0$
$$
\Omega(s)\le C\, e^{-\theta s}\;,
$$
it follows that 
\begin{align*}
& R\big(\mu\big(B_{r}(x)\big),N,p,M\big)\le \\
& C \bigg[\big(1+[\log(r^{-1})]\big)
\; \bigg(r + r^{-n-1}(1+v\mathfrak{p}_{0} \log (r^{-1}))^2\ \exp{\big(-2\theta\mathfrak{p}_{0}\log(r^{-1})}\big)\\
&
\quad + r^{-2n-21}(1+v\mathfrak{p}\log (r^{-1}))^2\ \exp{\big(-2\theta \mathfrak{p}\log(r^{-1})}\big)\\
&
\quad +r^{-3\big(n+2+(v+3)\mathfrak{p}_{0}\log A\big)}\;\exp{\Big(-4\theta \big(n+2+(v+3)\mathfrak{p}\log
 A\big)\;\frac{\log r}{\log\alpha}\Big)} \\
&
\quad +(\log (r^{-1}))^{3}\ \exp{ \left(-\theta\frac{1}{3}\;\min\left\{\frac14,\; \frac{-\log\alpha}{24\log A}\right\}\log (r^{-1})\right)}\\
&
\quad +\max\left\{r^{\frac12},\alpha^{\frac{\log (r^{-1})}{12\log A}}\right\}
+ 4\big(\log(r^{-1})+n+2\\
&\quad+(v+3)\mathfrak{p}\log A\big)\;\frac{\log r}{\log\alpha}\;
\big( r^{\frac{\mathfrak{r}}{2}} +r^{-3-n}\; \exp\big(-2\theta(\log (r^{-1}))^{2}\big)
\bigg)\\
&
\quad+\big(1+[\log(r^{-1})]\big)^{2}\mu\big(B_{r}(x)\big)\\
& \quad+
\frac{t^{1+[\log(r^{-1})]}}{\big(1+[\log(r^{-1})]\big)!} 
+ \mu\big(B_{r}(x)\big)
\bigg]\;.
\end{align*}
We now take $\mathfrak{p}_{0}$ large enough so that for any $r\in(0,1/2)$
$$
r^{-n-1}(1+v\mathfrak{p}_{0} \log (r^{-1}))^2\
 \exp{\big(-2\theta\mathfrak{p}_{0}\log(r^{-1})}\big)\le r\;.
$$
We then choose $\mathfrak{p}$ large enough so that for any $r\in(0,1/2)$
$$
r^{-2n-21}(1+v\mathfrak{p}\log (r^{-1}))^2\ \exp{\big(-2\theta \mathfrak{p}\log(r^{-1})}\big)
$$
$$
+r^{-3\big(n+2+(v+3)\mathfrak{p}_{0}\log A\big)}\;\exp{\big(-4\theta \big(n+2+(v+3)\mathfrak{p}\log A\big)\;\frac{\log r}{\log\alpha}\big)} \le r\;.
$$
We obtain
$$
R\big(\mu\big(B_{r}(x)\big),N,p,M\big)\le C \;r^{a}
$$
for some constant $a>0$. 
Similarly, choosing $\mathfrak{a}=1/(3\log A)$ there exists a constant
$b>0$ such that
$$
\mu(\trestresmauvais_{r}) \leq C\;r^{b}\;.
$$
Theorem \ref{nuhpoisson} now follows from Theorem \ref{abstract}.

\appendix

\section{Some consequences of Besicovitch covering Lemma}

We state and prove a few lemmas which result from a version of Besicovitch's covering Lemma
valid on Riemannian manifolds \cite[Section 2.8]{federer}. Some of these lemmas
may be useful in more general contexts.

\begin{lemma}\label{mesbou}
Let $\mu$ be a probability measure with compact support in
a $n$-dimensional Riemannian manifold $M$. Then, for any $\mathfrak{g}>0$, there exists a constant $C>0$ such that
for any $r\in]0,1]$ 
$$
\mu\left(\left\{x\,\big|\, 0<\mu\big(B_{r}(x)\big)\le
r^{n+\mathfrak{g}}\right\}\right)\le C r^{\mathfrak{g}}\;.
$$ 
\end{lemma}

\begin{proof}
Let 
$$
\mathcal{F}_{r}= \left\{x\,\big|\, \mu\big(B_{r}(x)\big)\le r^{n+\mathfrak{g}}\right\}\;.
$$
The family of balls $\mathcal{C}=\{B_{r}(x)\,:\,x\in \mathcal{F}_{r}\} $
is obviously a covering of $\mathcal{F}_{r}$. Therefore, by Besicovitch's covering
Lemma,  there is a finite number $p(n)$ and $q$ collections of balls belonging to  $\mathcal{C}$,
denoted by $\mathcal{H}_{1},\ldots,\mathcal{H}_{q}$, with $q\le p(n)$,
such that in each collection $\mathcal{H}_{i}$ the balls are pairwise
disjoint, and the collection of all the balls in all the $\mathcal{H}_{i}$
($1\le i\le q$) cover $\mathcal{F}_{r}$.
We have
\begin{eqnarray*}
\mu\big(\mathcal{F}_{r}\big) & \le & 
\sum_{i=1}^{q}\sum_{B\in\mathcal{H}_{i},\,\mu(B)>0} \mu(B)
\\
& \le & \sum_{i=1}^{q}\sum_{B\in\mathcal{H}_{i},\,\mu(B)>0} r^{n+\mathfrak{g}}. 
\end{eqnarray*}
Since $\mu$ has compact support, there is a number $R_0>0$ such that
$$
\bigcup_{i=1}^{q}\bigcup_{B\in
  \mathcal{H}_{i},\,\mu(B)>0}B\subset B_{R_0}(0)\;.
$$
Therefore, since the balls in each $\mathcal{H}_{i}$ are disjoint, there is a constant $C'$ such that for any $1\le i\le q$ we
have
$$
\card\big(\big\{B\in \mathcal{H}_{i}\,\big|\,\mu(B)>0\big\}\big)\le
C'\,r^{-n}\;. 
$$
This implies
$$
\mu\big(\mathcal{F}_{r}\big)\le p(n)\,C'\,r^{\mathfrak{g}}\;.
$$
\end{proof}

\begin{lemma}\label{besico1}
Let $\mu$ be a Borel probability measure on a $n$-dimensional Riemannian manifold $M$. 
For $r>0$ and $\mathfrak{s}>0$ define
$$
\mathscr{E}_{r,\mathfrak{s}}=\big\{x\,\big|\, \mu\big(B_{2r}(x)\big)> r^{-\mathfrak{s}}
\mu\big(B_{r}(x)\big) \big\}\;. 
$$
There is a constant $C>0$ independent of $r$ and $\mathfrak{s}$ (it depends
only on $n$) such that
$$
\mu\big(\mathscr{E}_{r,\mathfrak{s}}\big)\le C\,r^{\mathfrak{s}}\;.
$$
\end{lemma}
\begin{proof}
The family of balls $\mathcal{C}=\{B_{r}(x):x\in \mathscr{E}_{r}\} $ is obviously
a covering of $\mathscr{E}_{r,\mathfrak{s}}$. Therefore 
by the Besicovitch covering Lemma, 
there is a finite number $p(n)$ and $q$ collections of balls belonging to  $\mathcal{C}$,
$\mathcal{H}_{1},\ldots,\mathcal{H}_{q}$ with $q\le p(n)$
such that in each collection $\mathcal{H}_{i}$ the balls are pairwise
disjoint, and the collection of all the balls in all the $\mathcal{H}_{i}$
($1\le i\le q$) cover $\mathscr{E}_{r,\mathfrak{s}}$.
For any $1\le i\le q$, we will denote by $\mathcal{K}_{i}$  the set 
of centers of the balls in $\mathcal{H}_{i}$.

For any $1\le i\le q$, we consider the set of balls   
$\mathcal{C}_{i}=\{B_{2r}(x)\,:\,x\in
\mathcal{K}_{i}\} $. This is obviously a covering of $\mathcal{K}_{i}$
and the main observation is that  each point is covered by only one ball. 
Indeed, if  some $x\in \mathcal{K}_{i}$, belongs to a ball 
$B_{2r}(y)$ with $y\in \mathcal{K}_{i}$, then $d(y,x)\le 2r$ which
implies $y=x$ since otherwise $B_{r}(x)\cap B_{r}(y)\neq\emptyset$. 

Applying once more the Besicovitch Lemma to the covering $\mathcal{C}_{i}$
of $\mathcal{K}_{i}$,  we conclude that there exists 
$q_{i}\le p(n)$ 
collections $\mathcal{H}_{i,1},\ldots,\mathcal{H}_{i,q_{i}}$
of pairwise disjoint balls
 of  $\mathcal{C}_{i}$ such that each collection is  at most countable 
and  the union of all the balls in all these $q_{i}$ collections 
covers $\mathcal{K}_{i}$.

For any $1\le i\le q$ and $1\le \ell\le q_{i}$ we have
$$
\sum_{B\in \mathcal{H}_{i,\ell}}\mu(B)=\mu\left(\bigcup_{B\in
  \mathcal{H}_{i,\ell}}B\right)\le 1\;
$$
which implies 
$$
\sum_{i=1}^{q}\sum_{\ell=1}^{q_{i}}\sum_{B\in
  \mathcal{H}_{i,\ell}}\mu(B)\le p(n)^{2}\;.
$$
Since 
$$
\mathscr{E}_{r,\mathfrak{s}}\subset \bigcup_{i=1}^{q} \bigcup_{x\in 
\mathcal{K}_{i}}B_{r}(x)\;,
$$
we have  
$$
\mu\big(\mathscr{E}_{r,\mathfrak{s}}\big)\le \sum_{i=1}^{q}\sum_{x\in 
\mathcal{K}_{i}}\mu\big(
B_{r}(x)\big)\;.
$$
From the definition of $\mathscr{E}_{r,\mathfrak{s}}$ we get
$$
\mu\big(\mathscr{E}_{r,\mathfrak{s}}\big)\le r^{\mathfrak{s}} \sum_{i=1}^{q}
\sum_{x\in \mathcal{K}_{i}}\mu\big( B_{2r}(x)\big)\le
 r^{\mathfrak{s}} \sum_{i=1}^{q} \sum_{\ell=1}^{q_{i}}\sum_{B\in
  \mathcal{H}_{i,\ell}}\mu(B)
\le r^{\mathfrak{s}}\;p(n)^{2} \;.
$$ 
This finishes the proof of the Lemma with $C=p(n)^{2}$. 
\end{proof}

\begin{lemma}\label{besicodeux}
Let $\lambda_0$ and $\lambda_1$ be two finite positive measures on a $n$-dimensional Riemannian manifold $M$.
For $\omega\in(0,1)$ and $r\in(0,1)$, define the set
$$
\mathscr{C}_{\omega}(\lambda_0,\lambda_1,r)=\left\{x\in M\,\big|\,\lambda_1\big(B_{r}(x)\big)\ge
  \omega \lambda_{0}\big(B_{r}(x)\big)\right\}\;.
$$
There is an integer $p(n)$ such that  
$$
\lambda_{0}\big(\mathscr{C}_{\omega}(\lambda_0,\lambda_1,r)\big)\le p(n) \;\omega^{-1} 
\lambda_1(M).
$$
\end{lemma}

\begin{proof}
The family of balls $\mathcal{D}=\{B_{r}(x)\,:\,x\in \mathscr{C}_{\omega}(\lambda_0,\lambda_1,r)\}$
is obviously a covering of $\mathscr{C}_{\omega}(\lambda_0,\lambda_1,r)$.
Therefore, by the Besicovitch covering Lemma,
there is a finite number $p(n)$ and $q$ collections
of balls belonging to  $\mathcal{D}$, denoted by
$\mathcal{H}_{1},\ldots,\mathcal{H}_{q}$, with $q\le p(n)$, and
such that in each collection $\mathcal{H}_{i}$ the balls are pairwise
disjoint, and the collection of all the balls in all the $\mathcal{H}_{i}$
($1\le i\le q$) cover $\mathscr{C}_{\omega}(\lambda_0,\lambda_1,r)$.
For any $1\le i\le q$, we will denote by $\mathcal{K}_{i}$  the set 
of centers of the balls in $\mathcal{H}_{i}$.
Therefore, since the balls in each family are disjoint, we get
\begin{eqnarray*}
\lambda_{0}\big(\mathscr{C}_{\omega}(\lambda_0,\lambda_1,r)\big) 
& \le &
\sum_{i=1}^{q}\sum_{x\in \mathcal{K}_{i}}\lambda_{0}\big(B_{r}(x)\big)
\\
& \le &
\omega^{-1}\sum_{i=1}^{q}\sum_{x\in \mathcal{K}_{i}} \lambda_{1}\big(B_{r}(x)\big)
\\
& \le &
\omega^{-1}\;p(n)\;\lambda_{1}\big(M)\;.
\end{eqnarray*}
\end{proof}

The following corollary holds under the notations of Section \ref{NUDS}.
Its proof is an immediate consequence of the previous lemma.

\begin{corollary}\label{besico2mu}
For any non-negative integer $q$, let $\mu_{q}$ be the measure defined by
$$
\mu_{q}(A)=\sum_{i,R_{i}\ge q+1}^{\infty}
\sum_{j=0}^{R_{i}-1}m\big(T^{-j}\big(A\big)\cap
 \Lambda_{i}\big)\;. 
$$
Note that $\mu_{0}=\mu$, the SRB measure. For $\omega\in(0,1)$ and $r\in(0,1)$, define the set
$$
\mathscr{C}_{\omega}=\left\{x\in\attrac\,\big|\,\mu_{q}\big(B_{r}(x)\big)\ge
  \omega \mu_{0}\big(B_{r}(x)\big)\right\}\;.
$$
There is an integer $p(n)$ such that  
$$
\mu_{0}\big(\mathscr{C}_{\omega}\big)\le p(n) \;\omega^{-1} 
 \sum_{i,R_{i}\ge q+1}R_{i}\;m\big(\Lambda_{i}\big)\;.
$$
\end{corollary}

\section{Some technical estimates}

The following lemmas hold under the notations and the assumptions of Section \ref{NUDS}.

\begin{lemma}\label{lespieces}
There is a constant $C>0$ such that for any $\gamma\in \legu$ and any $i$,
we have
$$
\leb_{\gamma}\big(\Lambda_{i}\big)\ge C \ A^{-n_{u}R_{i}}
$$
and
$$
m_{\gamma}\big(\Lambda_{i}\big)\ge C B^{-1}\ A^{-n_{u}R_{i}}\ ,
$$
where $A$ is the constant defined in \eqref{leA} and $B$ is the constant appearing in \eqref{densite}.
\end{lemma}
\begin{proof}
From the Markov property, it follows that $T^{R_{i}}\big(\gamma\cap
\Lambda_{i}\big)=\gamma'\cap \Lambda $ for some $\gamma'\in\legu$. Since the
Jacobian of $T$ is bounded above by $A^{n_{u}}$, we have
$$
 A^{n_{u}R_{i}}\;\leb_{\gamma}\big(\Lambda_{i}\big)\ge 
\leb_{\gamma'}\big(\Lambda\big)\;.
$$ 
By the distorsion property of the Jacobian along the stable holonomy
(see property (P5)(b) in section \ref{NUDS}), there is a constant
$D>1$ such that for any $\gamma''\in\legu$ we have
$$
D^{-1}\leb_{\gamma''}\big(\Lambda\big)\le \leb_{\gamma'}\big(\Lambda\big)\le D
\leb_{\gamma''}\big(\Lambda\big)\;. 
$$
It follows immediately from \eqref{munormalise} that there is a constant
$D'>0$ such that 
$$
\inf_{\gamma''}\leb_{\gamma''}\big(\Lambda\big)\ge D'\;.
$$ 
The first estimate of the lemma follows. The second estimate follows from \eqref{couic}
and \eqref{densite}.
\end{proof}

\begin{lemma}\label{mesureboule}
There exist two constants $C>0$, $\mathfrak{r}>0$ and, for any $r\in(0,1)$, there exists a measurable set $\mathcal{J}_r$
such that
$$
\mu\big(\mathcal{J}_r\big)\leq C\ \Omega(\log (r^{-1})/(4\log A))
$$
and for any $x\in\attrac\backslash \mathcal{J}_r$ we have
$$
\mu\big(B_{r}(x)\big)\le C \ r^\mathfrak{r}.
$$
\end{lemma}

\begin{proof}
Let $\mathfrak{r}'>0$ to be chosen later on. We have
$$
\mu\big(B_{r}(x)\big)=
\sum_{i=1}^\infty\sum_{j=0}^{R_{i}-1}m\big(T^{-j}
\big(B_{r}(x)\big)\cap \Lambda_{i}\big)
$$
$$
=
\sum_{i,\, R_{i}< \mathfrak{r}'\log (r^{-1})}\sum_{j=0}^{R_{i}-1}m\big(T^{-j}
\big(B_{r}(x)\big)\cap \Lambda_{i}\big)
+ \mu_1\big(B_{r}(x)\big)
$$
where
$$
\mu_1\big(A\big)=
\sum_{i,\, R_{i}\geq \mathfrak{r}'\log (r^{-1})}\sum_{j=0}^{R_{i}-1}m\big(T^{-j}
\big(A\big)\cap \Lambda_{i}\big).
$$
Since $T$ is a diffeomorphism we have (see \eqref{leA})
$$
m\big(T^{-j}
\big(B_{r}(x)\big)\cap \Lambda_{i}\big)\le
m\big(B_{2A^{j}r}(y)\cap \Lambda_{i}\big)
$$
for some $y\in \Lambda_{i}$. 
Using \eqref{densite} we have
$$
m_{w}\big(B_{2A^{j}r}(y)\cap\Lambda_{i}\big)\leq B\ \leb_w\big(B_{2A^{j}r}(y)\big)\leq B\ (2A^{j}r)^{\dimu},
$$
and by \eqref{desintm} and Lemma \ref{lespieces} this implies 
\begin{eqnarray*}
m\big(B_{2A^{j}r}(y)\cap \Lambda_{i}\big)
& = & \int\d\nu(w)\  m_{w}\big(B_{2A^{j}r}(y)\cap \Lambda_{i}\big)\\
&=& \int\d\nu(w)\ \frac{m_{w}\big(B_{2A^{j}r}(y)\cap\Lambda_{i}\big)}{m_{w}\big(\Lambda_{i}\big)}\
m_{w}\big(\Lambda_{i}\big)\\
& \le & \Oun\ r^{\dimu (1-2\mathfrak{r}'\log A)}\int\d\nu(w)m_{w}\big(\Lambda_{i}\big)\\
& = & \Oun\ r^{\dimu (1-2\mathfrak{r}'\log A)} 
m\big( \Lambda_{i}\big).
\end{eqnarray*}
We choose $\mathfrak{r}'=1/(4\log A)$ and $\mathfrak{r}=\dimu/2$.
To finish the proof we apply Corollary \ref{besico2mu} with $q=\mathfrak{r}'\log (r^{-1})+1$ and
$\omega=\Omega(\mathfrak{r}'\log (r^{-1}))$. This finishes the proof.
\end{proof}

\begin{lemma}\label{itervide} For any given integer $k$, for any integer
  $p$,  and any $r>0$ we have
$$
\left\{x\;\bigg|\; B_{r}(x)\cap
T^{k}\big(B_{r}(x)\big)\neq\emptyset\right\}
\subset  \left\{x\;\bigg|\; B_{s_{p}r}(x)\cap
T^{2^{p}k}\big(B_{s_{p}r}(x)\big)\neq\emptyset\right\}\;.
$$
where
$$
s_{p}=2^{p}\;\frac{A^{k2^{p}}-1}{A^{k}-1}\cdot
$$
\end{lemma}
\begin{proof}
We first consider the case $p=1$. 
Let $x$ be such that $B_{r}(x)\cap
T^{k}\big(B_{r}(x)\big)\neq\emptyset$. This implies
$T^{k}\big(B_{r}(x)\big)\cap
T^{2k}\big(B_{r}(x)\big)\neq\emptyset$. Moreover there exists $z\in
B_{r}(x)$  such that $T^{k}(z)\in B_{r}(x)$. For any $u\in
T^{k}\big(B_{r}(x)\big)$, there is a  $v\in
B_{r}(x)$  such that $T^{k}(v)=u$. Therefore 
$$
d\big(u,T^{k}(z)\big)=d\big(T^{k}(v),T^{k}(z)\big)\le A^{k}d(v,z)
\le 2 A^{k}r\;.
$$
This implies by the triangle inequality
$$
d(u,x)\le d\big(u,T^{k}(z)\big)+d\big(x,T^{k}(z)\big)\le \big(2A^{k}+2\big)r\;.
$$
In other words $T^{k}\big(B_{r}(x)\big)\subset B_{(2A^{k}+2)r}(x)$. 
From the obvious inclusion
$T^{2k}\big(B_{r}(x)\big)\subset T^{2k}\big(B_{(2A^{k}+2)r}(x)\big)$, the
case $p=1$ follows, namely
\begin{align*}
& \left\{x\;\bigg|\; B_{r}(x)\cap
T^{k}\big(B_{r}(x)\big)\neq\emptyset\right\}\\
&
\qquad\qquad\subset  \left\{x\;\bigg|\; B_{2(A^{k}+1)r}(x)\cap
T^{2k}\big(B_{2(A^{k}+1)r}(x)\big)\neq\emptyset\right\}\;.
\end{align*}
The general case follows by induction.  
\end{proof}

\begin{lemma}
There exists a   constant $0<r_0<1$ such that for all $r\in (0,r_0)$, for all
$i$ such that $R_i\leq (\log (r^{-1}))/(4\log A)$,
 for all $0\leq j <R_i$, for every $x\in\attrac$, 
for any $\gamma_0\in \legu$, we have
$$
\textup{Card}\big\{
\gamma\in\legu \big| \gamma\cap \Lambda \subset T^{R_i}(\gamma_0\cap \Lambda_i)
\;\;\textup{and}\;\;\gamma\cap \Lambda \cap T^{R_i-j}(B_r(x))\neq\emptyset\big\}\leq 1.
$$
\end{lemma}

\begin{proof}
Let us assume that the above cardinality is greater than one. So let $\gamma_1\neq\gamma_2$
with
$$
\gamma_1,\gamma_2\in
\big\{
\gamma\in\legu \big| \gamma\cap \Lambda \subset T^{R_i}(\gamma_0\cap \Lambda_i)
\;\;\textup{and}\;\;\gamma\cap \Lambda \cap T^{R_i-j}(B_r(x))\neq\emptyset\big\}.
$$
Let $M_1\in T^{j-R_i}(\gamma_1\cap\Lambda)\cap B_r(x)$ and 
$M_2\in T^{j-R_i}(\gamma_2\cap\Lambda)\cap B_r(x)$.
Since $M_{1}$ and $M_{2}$ belong to the ball $B_r(x)$ we have
$$
d\big(T^{R_i-j}(M_{1}),T^{R_i-j}(M_{2})\big)\le
A^{R_i-j}\;d\big(M_{1},M_{2}\big)\le 2r^{3/4} \;.
$$
Let $P=\gamma_{2}\cap \gamma^{s}(T^{R_i-j}(M_{1}))\cap\Lambda$ (there is
one and only one such point by property 3) of $\legu$ and $\legs$).
Since the elements of  $\legu$ and $\legs$ are uniformly embedded
regular disks  with angles bounded away from zero (cf. property 2) of $\legu$ and $\legs$),
we conclude that there is a constant $C>0$ such that uniformly in $r$ small enough,
$M_{1}$ and $\gamma_{2}$,  we have
$$
d\big(T^{R_i-j}(M_{1}),P\big)\le 2\,C\,r^{3/4}\;.
$$
Let
$$
D_{0}=\gamma^{s}\big(T^{-j}(M_{1})\big)
\cap B_{4\,C\,r^{1/3}}\big(T^{-j}(M_{1})\big)\, .
$$
There exists a constant $C'>0$ such that uniformly in $r$ small enough
and in $M_{1}$ we have by \eqref{leA}
$$
T^{R_{i}}\big(D_{0}\big)\supset 
\gamma^{s}\big(T^{R_{i}-j}(M_{1})\big)
\cap B_{C'\, r^{5/6}}\big(T^{R_{i}-j}(M_{1})\big)\;.
$$
Hence, for a uniform $r_0$ small enough and any $r<r_0$, $P\in T^{R_{i}}\big(D_{0}\big)$.
This implies $T^{-R_{i}}(P)$ and $T^{-j}(M_{1})$ belong to
$\gamma_0\cap D_{0}$. This is a contradiction with
property 3) of $\legu$ and $\legs$, and the lemma is proved. 
\end{proof}

\begin{lemma}\label{lev}
Let
$$
\mathscr{V}_{p}(x,r)=\bigcup_{i}\bigcup_{j=0}^{R_{i}-1}
\bigcup_{\stackrel{\gamma\in\legs}{T^{p+j}(\gamma\cap\Lambda_{i})\cap
\partial B_{r}(x)\neq \emptyset}}T^{j}\big(\gamma\cap\Lambda_{i}\big).
$$
Then for all $x\in \attrac$, for all $r\in(0,1)$ and for all $p$, we have
$$
\mu\big(\mathscr{V}_{p}(x,r)\big)
\le \mu\big(\widetilde{\mathscr{C}}_{r,p\log\alpha/\log r}(x)\big)
\;,
$$ 
where $\widetilde{\mathscr{C}}_{r,p\log\alpha/\log r}(x)$ is the corona
$$
\widetilde{\mathscr{C}}_{r,p\log\alpha/\log r}(x)= B_{r+\alpha^{p}}(x)\backslash B_{r-\alpha^{p}}(x)\;.
$$
\end{lemma}

\begin{proof}
If $T^{p+j}\big(\gamma\cap\Lambda_{i}\big)\cap B_r(x)\neq \emptyset$ then,
by the uniform contraction of stable manifolds (see condtion (P3) in section{NUDS}), we have
$$
T^{p+j}\big(\gamma\cap\Lambda_{i}\big)\subset 
\widetilde{\mathscr{C}}_{r,p\log\alpha/\log r}(x).
$$
Therefore
$$
\bigcup_{\stackrel{\gamma\in\legs}{T^{p+j}\big(\gamma\cap\Lambda_{i}\big)\cap
\partial B_{r}(x)\neq \emptyset}}T^{j}\big(\gamma\cap\Lambda_{i}\big)
\subset T^{-p}\big( \widetilde{\mathscr{C}}_{r,p\log\alpha/\log r}(x)\big),
$$
whence
$$
\mathscr{V}_{p}(x,r) 
\subset T^{-p}\big( \widetilde{\mathscr{C}}_{r,p\log\alpha/\log r}(x)\big).
$$
This implies by the invariance of $\mu$
$$
\mu\left(
\mathscr{V}_{p}(x,r) 
\right)\le 
\mu\big(\widetilde{\mathscr{C}}_{r,p\log\alpha/\log r}(x)\big)\;.
$$
\end{proof}



\begin{thebibliography}{99}

\bibitem{av}
M. Abadi and N. Vergne.
Sharp error for point-wise Poisson approximations in mixing processes.  Nonlinearity  {\bf 21} (2008), 2871--2885.

\bibitem{chenstein}
R. Arratia, L. Goldstein, L.  Gordon.
Two moments suffice for Poisson approximations: the Chen-Stein method.
Ann. Probab. {\bf 17} (1989), no. 1, 9--25. 

\bibitem{bstv}
H. Bruin, B. Saussol, S. Troubetzkoy and S. Vaienti.
Return time statistics via inducing. 
Ergodic Theory Dynam. Systems {\bf 23} (2003), no. 4, 991--1013. 

\bibitem{bv}
H. Bruin and S. Vaienti.
Return time statistics for unimodal maps. 
Fund. Math. {\bf 176} (2003), no. 1, 77--94. 


\bibitem{ccs}
J.-R. Chazottes, P. Collet, and B. Schmitt.
Statistical consequences of the Devroye inequality for processes.
Applications to a class of non-uniformly hyperbolic dynamical systems. 
Nonlinearity {\bf 18} (2005), no. 5, 2341--2364.

\bibitem{pierre}
P. Collet.
Statistics of closest return for some non-uniformly hyperbolic systems. 
Ergodic Theory Dynam. Systems {\bf 21} (2001), no. 2, 401--420. 

\bibitem{colletgalves}
P. Collet, A. Galves.
Statistics of close visits to the indifferent fixed point of an interval map. 
J. Statist. Phys. {\bf 72} (1993), no. 3-4, 459--478. 

\bibitem{pa}
P. Collet and A. Galves.
Asymptotic distribution of entrance times for expanding maps of the interval. 
In {\em Dynamical systems and applications}, 139--152,
World Sci. Ser. Appl. Anal., 4, World Sci. Publ., River Edge, NJ, 1995. 

\bibitem{dgs}
M. Denker, M. Gordin and A. Sharova.
A Poisson limit theorem for toral automorphisms. 
Illinois J. Math. {\bf 48} (2004), no. 1, 1--20. 

\bibitem{dimitry}
D. Dolgopyat.
Limit theorems for partially hyperbolic systems. 
Trans. Amer. Math. Soc. {\bf 356} (2004), no. 4, 1637--1689.

\bibitem{federer}
H. Federer.
{\em Geometric measure theory.}
Die Grundlehren der mathematischen Wissenschaften, Band 153. 
Springer-Verlag New York Inc., New York 1969.

\bibitem{fft}
A. C. Freitas, J. M. Freitas and M. Todd.
Hitting time statistics and extreme value theory. 
Probab. Th. \& Rel. Fields {\bf 147}(2010), nos. 3-4, 675--710.

\bibitem{ghn}
C. Gupta, M. Holland and M. Nicol.
Extreme value theory for dispersing billiards and a 
class of hyperbolic maps with singularities. Preprint, 2009.

\bibitem{hv1}
N. Haydn and S. Vaienti.
The limiting distribution and error terms for return times of dynamical systems. 
Discrete Contin. Dyn. Syst. {\bf 10} (2004), no. 3, 589--616. 

\bibitem{hirata1}
M. Hirata.
Poisson law for Axiom A diffeomorphisms.
Ergodic Theory Dynam. Systems {\bf 13} (1993), 533--556. 

\bibitem{hsv}
M. Hirata, B. Saussol and S. Vaienti.
Statistics of return times: a general framework and new applications. 
Comm. Math. Phys. {\bf 206} (1999), no. 1, 33--55. 

\bibitem{hnt}
M. Holland, M. Nicol, A. T\"or\"ok.
Extreme value distributions for non-uniformly hyperbolic dynamical systems.
To appear in Trans. AMS.

\bibitem{lecam}
L. Le Cam.
An approximation theorem for the Poisson binomial distribution.
Pacific J. Math. {\bf 10} (1960), 1181--1197. 

\bibitem{ps}
F. P\`ene and B. Saussol.
Back to balls in billiards. 
Communications in mathematical physics {\bf 293} (2010), no.3, 837--866.

\bibitem{young1}
L.-S. Young.
Statistical properties of dynamical systems with some
hyperbolicity.
Ann. of Math. (2) {\bf 147} (1998), no. 3, 585--650. 

\bibitem{young2}
L.-S. Young.
Recurrence times and rates of mixing.
Israel J. Math. {\bf 110} (1999), 153--188. 

\end{thebibliography}
\end{document}